\documentclass[12pt,leqno]{article}

\usepackage{amssymb,amsmath,amsthm}
\usepackage[all]{xy}

\usepackage{enumerate}
\usepackage{mathrsfs}
\usepackage{color}

\newcommand{\nc}{\newcommand}
\nc{\on}{\operatorname}

\newtheorem{theorem}{Theorem}[section]
\newtheorem{proposition}[theorem]{Proposition}
\newtheorem{lemma}[theorem]{Lemma}
\newtheorem{corollary}[theorem]{Corollary}

\theoremstyle{definition}

\newtheorem{definition}[theorem]{Definition}
\newtheorem{notation}[theorem]{Notation}
\newtheorem{example}[theorem]{Example}
\newtheorem{remark}[theorem]{Remark}

\nc{\RR}{\mathrm{R}}
\nc{\LL}{\mathrm{L}}

\newcommand{\C}{{\mathbb{C}}}
\newcommand{\N}{{\mathbb{N}}}
\newcommand{\R}{{\mathbb{R}}}

\newcommand{\Z}{{\mathbb{Z}}}

\newcommand{\BBT}{{\mathbb T}}
\newcommand{\BBP}{{\mathbb P}}

\def\phi{{\varphi}}
\def\epsilon{\varepsilon}

\newcommand{\cor}{{\bf k}}
\newcommand{\cort}{{\bf k}^\gamma}
\newcommand{\corg}{{\bf k}^\gamma}
\newcommand{\corgl}{{\bf k}^{\gamma,l}}
\newcommand{\corgr}{{\bf k}^{\gamma,r}}

\def\shd{\mathscr{D}}

\def\shl{\mathscr{L}}
\def\shm{\mathscr{M}}
\def\shn{\mathscr{N}}
\def\sho{\mathscr{O}}

\def\sht{\mathscr{T}}

\newcommand{\symx}{{\mathfrak{X}}}

\renewcommand{\ker}{\operatorname{Ker}}

\newcommand{\rmpt}{{\rm pt}}
\newcommand{\rmptt}{{\{\rm pt\}}}

\renewcommand{\to}[1][]{\xrightarrow[]{#1}}
\newcommand{\from}[1][]{\xleftarrow[]{#1}}
\newcommand{\isoto}[1][]{\xrightarrow[#1]%
{{\raisebox{-.6ex}[0ex][-.6ex]{$\mspace{1mu}\sim\mspace{2mu}$}}}}
\newcommand{\isofrom}[1][]{\xleftarrow[#1]%
{{\raisebox{-.6ex}[0ex][-.6ex]{$\mspace{1mu}\sim\mspace{2mu}$}}}}
\newcommand{\tto}{\rightrightarrows}

\newcommand{\muhom}{\mu hom}
\newcommand{\muHom}[1][]{\mathrm{Hom}^\mu_{\raise1.5ex\hbox to.1em{}#1}}
\newcommand{\Hom}[1][]{\mathrm{Hom}_{\raise1.5ex\hbox to.1em{}#1}}
\newcommand{\RHom}[1][]{\RR\mathrm{Hom}_{\raise1.5ex\hbox to.1em{}#1}}
\newcommand{\Ext}[2][]{\mathrm{Ext}_{\raise1.5ex\hbox to.1em{}#1}^{#2}}
\renewcommand{\hom}[1][]{{\mathscr{H}\mspace{-4mu}om}_{\raise1.5ex\hbox to.1em{}#1}}
\newcommand{\rhom}[1][]{{\RR\mathscr{H}\mspace{-3mu}om}_{\raise1.5ex\hbox to.1em{}#1}}
\newcommand{\rhomc}[1][]
{{\mathscr{H}\mspace{-3mu}om}^*_{\raise1.5ex\hbox to.1em{}#1}}

\newcommand{\ext}[2][]{{\mathscr{E}xt}_{\raise1.5ex\hbox to.1em{}#1}^{#2}}
\newcommand{\Tor}[2][]{\mathrm{Tor}^{\raise1.5ex\hbox to.1em{}#1}_{#2}}
\newcommand{\tenso}[1][]{\mathbin{\otimes_{\raise1.5ex\hbox to-.1em{}{#1}}}}
\newcommand{\ltens}[1][]{\mathbin{\overset{\mathrm{L}}\tenso}_{#1}}

\newcommand{\etens}{\mathbin{\boxtimes}}
\newcommand{\letens}{\overset{\mathrm{L}}{\etens}}

\newcommand{\Endo}[1][]{\mathrm{End}_{\raise1.5ex\hbox to.1em{}#1}}

\newcommand{\Aut}[1][]{\mathrm{Aut}_{\raise1.5ex\hbox to.1em{}#1}}

\newcommand{\rsect}{\mathrm{R}\Gamma}

\newcommand{\conv}[1][]{\mathop{\circ}\limits_{#1}}

\newcommand{\aconv}[1][]{\mathop{\circ}\limits^{a}\limits_{#1}}

\newcommand{\shnt}{\shn_{\mathrm{tor}}}

\newcommand{\oim}[1]{{#1}_*}

\newcommand{\roim}[1]{\RR{#1}_*}

\newcommand{\reim}[1]{\RR{#1}_!}
\newcommand{\opb}[1]{#1^{-1}}

\newcommand{\epb}[1]{#1^{!}}

\newcommand{\hplus}[1][]{\mathop{+}\limits_{#1}}
\newcommand{\eqdot}{\mathbin{:=}}

\newcommand{\cl}{\colon}
\newcommand{\scbul}{{\,\raise.4ex\hbox{$\scriptscriptstyle\bullet$}\,}}

\newcommand{\hatp}{\mathbin{\widehat +}}

\newcommand{\hatstar}{\mathbin{\widehat \star}}
\newcommand{\npstar}{\mathbin{\star_{np}}}
\newcommand{\tw}[1]{\widetilde{#1}}

\newcommand{\ol}{\overline}
\newcommand{\bl}{\bigl(}
\newcommand{\br}{\bigr)}
\newcommand{\lp}{{\rm(}}
\newcommand{\rp}{{\rm)}}

\newcommand{\piw}{{\widehat \pi}}

\newcommand{\ba}{\begin{array}}
\newcommand{\ea}{\end{array}}

\newcommand{\bnum}{\begin{enumerate}[{\rm(i)}]}
\newcommand{\enum}{\end{enumerate}}
\newcommand{\banum}{\begin{enumerate}[{\rm(a)}]}
\newcommand{\eanum}{\end{enumerate}}

\newcommand{\eq}{\begin{eqnarray}}
\newcommand{\eneq}{\end{eqnarray}}
\newcommand{\eqn}{\begin{eqnarray*}}
\newcommand{\eneqn}{\end{eqnarray*}}

\newcommand{\set}[2]{\left\{#1 \mathbin{;} #2 \right\}}

\def\rop{{\rm op}}
\def\op{{\rm op}}

\DeclareMathOperator{\id}{id}

\DeclareMathOperator{\supp}{supp}

\DeclareMathOperator{\chv}{char}

\newcommand{\Supp}{\on{Supp}}
\newcommand{\Der}[1][]{\mathsf{D}^{#1}}
\newcommand{\Derb}{\Der[\mathrm{b}]}

\newcommand{\Derlb}{\Der[\mathrm{lb}]}

\newcommand{\SSi}{\mathrm{SS}}
\newcommand{\RD}{\mathrm{D}}

\newcommand{\Int}{{\rm Int}}

\newcommand{\dT}{{\dot{T}}}
\newcommand{\dTM}{{\dT}^*M}

\newcommand{\LG}{{L_\gamma}}
\newcommand{\lG}{{l_\gamma}}
\newcommand{\RG}{{R_\gamma}}
\newcommand{\rG}{{r_\gamma}}

\newcommand{\indlim}[1][]{\mathop{\varinjlim}\limits_{#1}}

\newcommand{\sinddlim}[1][]{\smash{\mathop{``{\varinjlim}"}\limits_{#1}}\,}

\nc{\eps}{\varepsilon}
\nc{\hs}{\hspace*}
\nc{\nn}{\nonumber}
\nc{\tM}{\widetilde{M}}
\nc{\h}{\mathbf{h}}
\nc{\tf}{\tilde{f}}
\nc{\codim}{\on{codim}}
\nc{\lh}{\mathscr{H}}
\nc{\bwr}{\mbox{\large{$\wr$}}}
\nc{\dTi}{\dT^{*,\mathrm{in}}}
\nc{\Cd}{\mathrm{C}}

\begin{document}

\date{February 15, 2012}

\title{Microlocal theory of sheaves and Tamarkin's non displaceability theorem}
\author{St{\'e}phane Guillermou and Pierre Schapira}

\maketitle

\begin{abstract}
This paper is an attempt to better understand Tamarkin's approach of classical
non-displaceability theorems of symplectic geometry, based on the microlocal
theory of sheaves, a theory whose main features we recall here. 
If the main theorems are due to Tamarkin, our proofs may be rather different and in the course of the paper 
we introduce some new notions and obtain new results which may be of interest.  
\end{abstract}

\section*{Introduction}
In~\cite{Ta}, D.~Tamarkin gives a totally new approach for treating 
classical problems of non-displaceability in symplectic geometry. His
approach is based on the microlocal theory of sheaves, introduced 
and systematically developed in \cite{KS82,KS85,KS90}. (Note however that
the use of 
the microlocal theory of sheaves also appeared in a related context 
in~\cite{KO01,NZ09,N09}.)

The aim of this paper was initially to better understand Tamarkin's ideas and to
give more accessible proofs by making full use of the tools
of~\cite{KS90} and of the recent paper~\cite{GKS10}. But when working
on this subject,  we found some new results which may be of interest. In
particular, we make here a systematic study of the category of torsion objects.
\medskip

Let us first briefly recall the main facts of the microlocal theory of sheaves. 
Consider a real manifold $M$ of class $C^\infty$ and 
 a commutative unital ring $\cor$ of finite global dimension. Denote by 
$\Derb(\cor_M)$ the bounded derived category of sheaves of 
$\cor$-modules on $M$. In loc.\ cit.\, the authors attach to
an object $F$ of $\Derb(\cor_M)$
its singular support, or microsupport, $\SSi(F)$, 
a closed subset of $T^*M$, the cotangent bundle to $M$. The microsupport is
conic for the action of  $\R^+$ on $T^*M$ and is involutive 
({\em i.e.,} co-isotropic). 
The microsupport allows one to localize the triangulated category 
$\Derb(\cor_M)$, and in particular to define the category 
$\Derb(\cor_M;U)$ for an open subset $U\subset T^*M$. 
This theory is ``conic'', that is, it is invariant by the
$\R^+$-action and is related to the homogeneous
symplectic structure  rather than the symplectic structure. 

In order to get rid of the homogeneity, a classical trick is to add a variable which
replaces it. This trick appears for example in 
the complex case in~\cite{PS04} where a deformation quantization  ring 
(with an $\hbar$-parameter) is constructed on the cotangent bundle
$T^*X$ to a complex manifold $X$ 
by using the ring  of microdifferential operators of~\cite{SKK} on $T^*(X\times\C)$. 
Coming back to the real setting, denote by $t$ a coordinate on
$\R$, by $(t;\tau)$ the
associated coordinates on $T^*\R$, 
by  $T^*_{\{\tau>0\}}(M\times\R)$ 
the open subset $\{\tau>0\}$ of $T^*(M\times\R)$ and consider the map
\eqn
&&\rho\cl T^*_{\{\tau>0\}}(M\times\R)\to T^*M,\quad (x,t;\xi,\tau)\mapsto
(x;{\xi}/{\tau}).
\eneqn
Tamarkin's idea is to work in the localized category 
$\Derb(\cor_{M\times\R};\{\tau>0\})$, the localization of 
$\Derb(\cor_{M\times\R})$ by the triangulated subcategory 
$\Derb_{\{\tau\leq0\}}(\cor_{M\times\R})$ consisting of sheaves with microsupport contained
in the set $\{\tau\leq0\}$. He first proves the useful result which
asserts that this localized category is equivalent to the left orthogonal to 
$\Derb_{\{\tau\leq0\}}(\cor_{M\times\R})$ and that the convolution by the
sheaf $\cor_{\{t\geq0\}}$ is a projector on this left orthogonal. 

Let us introduce the notation $\Derb(\cort_{M})\eqdot\Derb(\cor_{M\times\R};\{\tau>0\})$
and, for a closed subset $A\subset T^*M$, let us denote by $\Derb_A(\cort_{M})$
the full triangulated subcategory  of $\Derb(\cort_{M})$  consisting of
objects with microsupport contained in $\opb{\rho}A$. 

The first result of Tamarkin is a separability theorem.
If $A$ and $B$
are two compact subsets of $T^*M$, $F\in\Derb_A(\cort_{M})$, $G\in\Derb_B(\cort_{M})$,  
and if $A\cap B=\emptyset$, then $\Hom[\Derb(\cort_{M})](F,G)\simeq0$. 

The second result of Tamarkin is a Hamiltonian isotopy invariance
theorem, up to torsion, that is, 
after killing what he calls the torsion objects.
An object $F\in\Derb(\cort_{M})$ is torsion if there exists $c\geq0$ such
that the natural map $F\to \oim{T_c}(F)$ is zero, $\oim{T_c}(F)$ denoting the image of
$F$ by the translation $t\mapsto t+c$ in the $t$-variable.
Let $I$ be an open interval of $\R$ containing $[0,1]$ and 
let $\Phi=\{\phi_s\}_{s\in I}$ be a Hamiltonian isotopy
(with $\phi_0=\id$) such that there exists a compact set $C\subset T^*M$
satisfying $\phi_s|_{T^*M \setminus C} = \id_{T^*M \setminus C}$
for all $s\in I$.
Tamarkin constructs a functor 
$\Psi\cl\Derb_A(\cort_{M})\to\Derb_{\phi_1(A)}(\cort_{M})$ such that 
$\Psi(F)$ is isomorphic to $F$ modulo torsion, for any
 $F\in\Derb_A(\cort_{M})$.

From these two results he easily deduces that if $A,B\subset T^*M$ are 
compact sets 
and if there exist $F\in\Derb_A(\cort_{M})$, $G\in\Derb_B(\cort_{M})$
such that the map 
$\RHom[\Derb(\cort_{M})](F,G)\to\RHom[\Derb(\cort_{M})](F,T_c(G))$ is not zero
for all $c\geq0$, 
then the sets $A$ and $B$ are mutually non displaceable, that is, for
any Hamiltonian isotopy $\Phi$ as above and any $s\in I$,
$A\cap\phi_s(B)\not=\emptyset$.

\medskip
Let us describe the contents of this paper.

In Section~\ref{section:mts} we recall some constructions and results
of~\cite{KS90} on the microlocal theory of sheaves. 

In Section~\ref{section:GKS} we recall
the main theorem of~\cite{GKS10} which allows one to quantize
homogeneous Hamiltonian isotopies and we also give some geometrical
tools relying homogeneous and non homogeneous symplectic geometry.  

In Section~\ref{section:conv}
we study convolution of sheaves on a trivial vector bundle
$E=M\times V$ over $M$  as well as the category 
$\Derb(\cor_{E};U_\gamma)$, the localization of the category $\Derb(\cor_{E})$ on 
$U_\gamma=E\times V\times\Int(\gamma_0^\circ)$ where $\Int(\gamma_0^\circ)$
is the interior of the polar cone to a closed convex proper cone
$\gamma_0$ in $V$. We prove in particular a separability theorem in
this category.

In Section~\ref{section:tam1} we introduce the Tamarkin category 
$\Derb(\cort_{M})$, that is, the category $\Derb(\cor_{E};U_\gamma)$
for $E=M\times\R$ and $\gamma_0=\{t\geq0\}$.

In Section~\ref{section:tam2} we make a systematic study of 
the category $\shnt$ of torsion objects,
proving that this category is triangulated and also
proving that, under some hypothesis on the microsupport,
an object is torsion if and only if its restriction to
one point is torsion (Theorem~\ref{th:propag-of-torsion}).

Finally, in Section~\ref{section:tam3} we give a  proof of the 
 Hamiltonian isotopy invariance theorem of Tamarkin. The existence
of the functor $\Psi$ mentioned above is now an easy consequence  
on the results of~\cite{GKS10}, and one checks that this
functor induces a functor isomorphic to the identity functor modulo torsion.
As already mentioned, Tamarkin's non displaceability
theorem is an easy corollary of the preceding results. 

\medskip

Note that, for the purposes we have in mind, we do not need to consider the
unbounded derived category $\Der(\cor_M)$, as did Tamarkin, but only its full
triangulated category $\Derlb(\cor_M)$ consisting of locally bounded
objects. Also note that our notations, as well as our proofs, may seriously
differ from Tamarkin's ones.

\medskip

In a next future, motivated by the papers of Fukaya-Seidel-Smith~\cite{FSS08} and
Nadler \cite{N09}, we plan to use the tools developed here to study sheaves associated with 
smooth Lagrangian manifolds.

\medskip
\noindent
{\bf Acknowledgment}
We have been very much stimulated by the interest of 
Claude Viterbo for the applications of sheaf theory to symplectic topology
and it is a pleasure to thank him here. We also thank Masaki Kashiwara for 
many enlightning discussions.

\section{Microlocal theory of sheaves}
\label{section:mts}
In this section, we recall some definitions and results from
\cite{KS90}, following its notations with the exception of  slight
modifications. We consider a real manifold $M$ of class $C^\infty$.

\subsubsection*{Some geometrical notions  (\cite[\S~4.2,~\S~6.2]{KS90})}
For a locally closed subset $A$ of $M$, one denotes by $\Int(A)$
its interior and by $\overline{A}$ its closure.
One denotes by $\Delta_M$ or simply $\Delta$ the diagonal of $M\times M$. 

One denotes by $\tau\cl TM\to M$ and  $\pi\cl T^*M\to M$ the tangent and
cotangent bundles to $M$. If $L\subset M$ is a (smooth) submanifold, we
denote by $T_LM$ its normal bundle and $T^*_LM$ its conormal
bundle. They are defined by the exact sequences 
\eqn
&&0\to TL\to L\times_MTM\to T_LM\to0,\\
&&0\to T^*_LM\to L\times_MT^*M\to T^*L\to0.
\eneqn
One identifies $M$ to $T^*_MM$, the zero-section of $T^*M$.
One sets $\dTM\eqdot T^*M\setminus T^*_MM$ and one denotes by
$\dot\pi_M\cl\dTM\to M$ the projection.

Let $f\cl M\to N$ be  a morphism of  real manifolds. 
To $f$ are associated the tangent morphisms
\eq\label{diag:tgmor}
\xymatrix{
TM\ar[d]^-\tau\ar[r]^-{f'}&M\times_NTN\ar[d]^-\tau\ar[r]^-{f_\tau}&TN\ar[d]^-\tau\\
M\ar@{=}[r]&M\ar[r]^-f&N.
}\eneq
By duality, we deduce the diagram: 
\eq\label{diag:cotgmor}
&&\xymatrix{
T^*M\ar[d]^-\pi&M\times_N\ar[d]^-\pi\ar[l]_-{f_d}\ar[r]^-{f_\pi}T^*N
                                      & T^*N\ar[d]^-\pi\\
M\ar@{=}[r]&M\ar[r]^-f&N.
}\eneq
One sets 
\eqn
&&T^*_MN\eqdot\ker f_d= \opb{f_d}(T^*_MM). 
\eneqn
Note that, denoting by $\Gamma_f$ the graph of $f$ in $M\times N$, 
the projection $T^*(M\times N)\to M\times T^*N$ identifies 
$T^*_{\Gamma_f}(M\times N)$ and $M\times_NT^*N$.

For two subsets $S_1,S_2\subset M$, their Whitney's normal cone,
denoted $C(S_1,S_2)$, is the closed cone of $TM$ defined  as
follows. Let $(x)$ be a local coordinate system 
and let $(x;v)$ denote the associated
coordinate system on $TM$. Then 
\eqn
&&\left\{
\parbox{60ex}{
$(x_0;v_0)\in C(S_1,S_2)\subset TM$ if and only
if there exists a sequence $\{(x_n,y_n,c_n)\}_n\subset S_1\times S_2\times\R^+$ such
that $x_n\to[n]x_0$, $y_n\to[n]x_0$  and $c_n(x_n-y_n)\to[n]v_0$.
}
\right.
\eneqn
For a subset $S$ of $M$ and a smooth closed submanifold $L$ of $M$, 
the Whitney's normal cone of $S$ along $L$, denoted $C_L(S)$,
is the image in $T_LM$ of $C(L,S)$. If $L=\{p\}$, we  write
$C_p(S)$ instead of $C_{\{p\}}(S)$.

Now consider the homogeneous symplectic manifold $T^*M$: it is endowed
with the Liouville $1$-form given in a local homogeneous symplectic coordinate system 
$(x;\xi)$ on $T^*M$ by
\eqn
&&\alpha_M=\langle\xi,dx\rangle.
\eneqn
The antipodal map $a_M$ is defined by:
\eq\label{eq:antipodal}
&&a_M\cl T^*M\to T^*M,\quad(x;\xi)\mapsto(x;-\xi).
\eneq
If $A$ is a subset of $T^*M$, we denote by $A^a$ instead of $a_M(A)$ 
its image by the antipodal map.

We shall use the Hamiltonian isomorphism 
$H\cl T^*(T^*M)\isoto T(T^*M)$ given in a local symplectic coordinate system 
$(x;\xi)$ by
\eqn
&&H(\langle\lambda,dx\rangle +\langle \mu,d\xi\rangle)
=-\langle\lambda,\partial_\xi\rangle +\langle \mu,\partial_x\rangle.
\eneqn
\begin{definition}{\rm (see~\cite[Def. 6.5.1]{KS90})}
\label{def:coisotropic}
A subset $S$ of $T^*M$ is co-isotropic \lp one also says involutive\rp\, 
at $p\in T^*M$ if for any
$\theta\in T^*_pT^*M$ such that the  Whitney normal cone $C_p(S,S)$ is
contained in the hyperplane $\{v\in TT^*M;\langle v,\theta\rangle=0\}$, 
one has $-H(\theta)\in C_p(S)$.  A set $S$ is co-isotropic if it is so
at each $p\in S$. 
\end{definition}
When $S$ is smooth, one recovers the usual notion.

\subsubsection*{Microsupport}
We consider a commutative unital ring $\cor$ of finite global dimension 
({\em e.g.} $\cor=\Z$).
We denote by $\Der(\cor_M)$ (resp.\ $\Derb(\cor_M)$) 
the derived category (resp.\  bounded derived category) of sheaves
of $\cor$-modules on $M$.

Recall the definition 
of the microsupport (or singular support) $\SSi(F)$ of a sheaf $F$.

\begin{definition}{\rm (see~\cite[Def.~5.1.2]{KS90})}
Let $F\in \Derb(\cor_M)$ and let $p\in T^*M$. 
One says that $p\notin\SSi(F)$ if there exists an open neighborhood
$U$ of $p$ such that for any $x_0\in M$ and any
real $C^1$-function $\phi$ on $M$ defined in a neighborhood of $x_0$ 
satisfying $d\phi(x_0)\in U$ and $\phi(x_0)=0$, one has
$(\rsect_{\{x;\phi(x)\geq0\}} (F))_{x_0}\simeq0$.
\end{definition}
In other words, $p\notin\SSi(F)$ if the sheaf $F$ has no cohomology 
supported by ``half-spaces'' whose conormals are contained in a 
neighborhood of $p$. 
\begin{itemize}
\item
By its construction, the microsupport is closed and is
$\R^+$-conic, that is,
invariant by the action of  $\R^+$ on $T^*M$. 
\item
$\SSi(F)\cap T^*_MM =\pi_M(\SSi(F))=\Supp(F)$.
\item
The microsupport satisfies the triangular inequality:
if $F_1\to F_2\to F_3\to[+1]$ is a
distinguished triangle in  $\Derb(\cor_M)$, then 
$\SSi(F_i)\subset\SSi(F_j)\cup\SSi(F_k)$ for all $i,j,k\in\{1,2,3\}$
with $j\not=k$. 
\end{itemize}
\begin{theorem}{\rm (see~\cite[Th.~6.5.4]{KS90})}
Let $F\in \Derb(\cor_M)$. Then its microsupport 
$\SSi(F)$ is co-isotropic.
\end{theorem} 
In the sequel, for a locally closed subset $Z$ in $M$, we denote by
$\cor_Z$ the constant sheaf with stalk $\cor$ on $Z$, extended by $0$
on $M\setminus Z$.
\begin{example}
(i) If $F$ is a non-zero local system on a connected manifold $M$,
then $\SSi(F)=T^*_MM$, the zero-section.

\noindent
(ii) If $N$ is a smooth closed submanifold of $M$ and $F=\cor_N$, then 
$\SSi(F)=T^*_NM$, the conormal bundle to $N$ in $M$.

\noindent
(iii) Let $\phi$ be $C^1$-function with $d\phi(x)\not=0$ when $\phi(x)=0$.
Let $U=\{x\in M;\phi(x)>0\}$ and let $Z=\{x\in M;\phi(x)\geq0\}$. 
Then 
\eqn
&&\SSi(\cor_U)=U\times_MT^*_MM\cup\{(x;\lambda d\phi(x));\phi(x)=0,\lambda\leq0\},\\
&&\SSi(\cor_Z)=Z\times_MT^*_MM\cup\{(x;\lambda d\phi(x));\phi(x)=0,\lambda\geq0\}.
\eneqn

\noindent
(iv) Let $(X,\sho_X)$ be a complex manifold and let $\shm$ be a coherent module
over the ring $\shd_X$ of holomorphic differential operators. (Hence,
$\shm$ represents a system of linear partial differential equations on
$X$.) Denote by $F=\rhom[\shd_X](\shm,\sho_X)$ the complex of
holomorphic solutions of $\shm$. Then $\SSi(F)=\chv(\shm)$, the
characteristic variety of $\shm$.
\end{example}

\subsubsection*{Functorial operations (proper and non-characteristic cases)}
Let $M$ and $N$ be two real manifolds. We denote by $q_i$ ($i=1,2$)
the $i$-th projection defined on $M\times N$ and by $p_i$ ($i=1,2$)
the $i$-th projection defined on $T^*(M\times N)\simeq T^*M\times T^*N$.

\begin{definition}
Let $f\cl M\to N$ be a morphism of manifolds and let 
$\Lambda\subset T^*N$ be a closed $\R^+$-conic subset. One says that 
$f$ is non-characteristic for $\Lambda$ \lp or else, $\Lambda$ 
is non-characteristic for $f$, or $f$ and $\Lambda$ are transversal\rp\, if
\eqn
&&\opb{f_\pi}(\Lambda)\cap T^*_MN\subset M\times_NT^*_NN.
\eneqn
\end{definition}
A morphism $f\cl M\to N$ is non-characteristic for a closed
$\R^+$-conic subset $\Lambda$ of $T^*N$ if and
only if $f_d\cl M\times_NT^*N\to T^*M$ is proper on $\opb{f_\pi}(\Lambda)$
and in this case 
$f_d\opb{f_\pi}(\Lambda)$ is closed and $\R^+$-conic in $T^*M$.

We denote by $\omega_M$ the dualizing complex on $M$.
Recall that $\omega_M$ is isomorphic to the orientation sheaf shifted by the
dimension. We also use the notation $\omega_{M/N}$ for the relative
dualizing complex $\omega_M\tenso\opb{f}\omega_N^{\tenso-1}$.
We have the duality functors
\eq\label{eq:dualfct}
&&\RD_M(\scbul)=\rhom(\scbul,\omega_M), \\
&&\RD'_M(\scbul)=\rhom(\scbul,\cor_M).
\eneq
\begin{theorem}\label{th:opboim}{\rm(See \cite[\S~5.4]{KS90}.)}
Let $f\cl M\to N$ be a morphism of manifolds,
let $F\in\Derb(\cor_M)$ and let $G\in\Derb(\cor_N)$.  
\bnum
\item
One has
\eqn
&&\SSi(F\letens G)\subset\SSi(F)\times\SSi(G),\\
&&\SSi(\rhom(\opb{q_1}F,\opb{q_2}G))\subset\SSi(F)^a\times\SSi(G).
\eneqn 
\item
Assume that $f$ is proper on $\Supp(F)$. Then
$\SSi(\reim{f}F)\subset f_\pi\opb{f_d}\SSi(F)$.
\item
Assume that $f$ is non-characteristic with
respect to $\SSi(G)$. Then 
the natural morphism $\opb{f}G\tenso \omega_{M/N}\to\epb{f}(G)$
is an isomorphism. Moreover
$\SSi(\opb{f}G) \cup \SSi(\epb{f}G) \subset f_d\opb{f_\pi}\SSi(G)$.
\item
Assume that $f$ is smooth
 \lp that is, submersive\rp. Then $\SSi(F)\subset M\times_NT^*N$
if and only if, for any $j\in\Z$, the sheaves $H^j(F)$
are  locally constant on the fibers of $f$.
\enum
\end{theorem}
For the notion of a cohomologically constructible sheaf we refer
to~\cite[\S~3.4]{KS90}.
\begin{corollary}\label{cor:opboim}
Let $F_1,F_2\in\Derb(\cor_M)$.
\bnum
\item
Assume that $\SSi(F_1)\cap\SSi(F_2)^a\subset T^*_MM$. Then
\eqn
&&\SSi(F_1\ltens F_2)\subset \SSi(F_1)+\SSi(F_2).
\eneqn
\item
Assume that $\SSi(F_1)\cap\SSi(F_2)\subset T^*_MM$. Then 
\eqn
&&\SSi(\rhom(F_1,F_2))\subset \SSi(F_1)^a+\SSi(F_2).
\eneqn
Moreover, assuming that $F_1$ is cohomologically constructible,
the natural morphism $\RD'F_1 \ltens F_2 \to \rhom(F_1,F_2)$
is an isomorphism.
\enum
\end{corollary}
The next result follows immediately from Theorem~\ref{th:opboim}~(ii). It is
a particular case of the microlocal Morse lemma
(see~\cite[Cor.~5.4.19]{KS90}), the classical theory corresponding to the
constant sheaf $F=\cor_M$.
\begin{corollary}\label{cor:Morse}
Let $F\in\Derb(\cor_M)$, 
let $\phi\cl M\to\R$ be a function of class $C^1$ and assume that $\phi$
is proper on $\supp(F)$. 
Let $a<b$ in $\R$
and assume that $d\phi(x)\notin\SSi(F)$ for $a\leq \phi(x)<b$. Then the
natural morphism \\
$\rsect(\opb{\phi}(]-\infty,b[);F)\to\rsect(\opb{\phi}(]-\infty,a[);F)$
is an isomorphism.
\end{corollary}

\begin{corollary}\label{cor:opbeqv}
Let $I$ be a contractible manifold and let $p\cl M\times I\to M$ be the projection.
If $F\in\Derb(\cor_{M\times I})$ satisfies  $\SSi(F)\subset T^*M\times T^*_II$,
then $F\simeq\opb{p}\roim{p}F$.
\end{corollary}
\begin{proof}
It follows from Theorem~\ref{th:opboim}~(iv) that the restriction
$F\vert_{\{x\} \times I}$ is locally constant for any $x\in M$. Then the
result follows from~\cite[Prop.~2.7.8]{KS90}.
\end{proof}
\begin{corollary}\label{cor:rsectFt}
Let $I$ be an open interval of $\R$ and let $q\cl M\times I\to I$ be the projection.
Let $F\in\Derb(\cor_{M\times I})$ such that 
$\SSi(F) \cap (T^*_MM \times T^*I) \subset T^*_{M\times I}(M\times I)$ and 
$q$ is proper on $\Supp(F)$.
Then we have isomorphisms 
$\rsect(M;F_s) \simeq \rsect(M;F_t)$ for any $s,t\in I$.
\end{corollary}
\begin{proof}
It follows from Theorem~\ref{th:opboim} that
$\SSi(\roim{q}(F))\subset T^*_II$. Hence, there exists
$V\in\Derb(\cor)$ and an isomorphism
$\roim{q}(F) \simeq V_I$.
(Recall that $V_I=\opb{a_I}V$, where $a_I\to\rmptt$ is the projection and $V$
is identified to a sheaf on $\rmptt$.)  
Since we have
$\rsect(M; F_s) \simeq (\roim{q}(F))_s$
the result follows.
\end{proof}

\subsubsection*{Kernels (\cite[\S~3.6]{KS90})}
\begin{notation}\label{not:Mijk}
Let $M_i$ ($i=1,2,3$) be manifolds. For short, we write $M_{ij}\eqdot M_i\times M_j$ 
($1\leq i,j\leq3$) and $M_{123}=M_1\times M_2\times M_3$. 
We denote by $q_i$  the projection $M_{ij}\to M_i$ or the projection $M_{123}\to M_i$
and by $q_{ij}$ the projection $M_{123}\to M_{ij}$. Similarly, we denote by
$p_i$  the projection $T^*M_{ij}\to T^*M_i$ or the projection $T^*M_{123}\to T^*M_i$
and by $p_{ij}$ the projection $T^*M_{123}\to T^*M_{ij}$. We also need
to introduce the map $p_{12^a}$, the composition of $p_{12}$ and the
antipodal map on $T^*M_2$. 

Let $A\subset T^*M_{12}$ and $B\subset T^*M_{23}$. We set
\eq\ba{rcl}\label{eq:convolution_of_sets}
&&A\times_{T^*M_{2^a}}B=\opb{p_{12}}(A)\cap\opb{p_{2^a3}}(B)\\
&&A\aconv B=p_{13}(A\times_{T^*M_{2^a}}B)\\
&&\hspace{6.1ex}=\{(x_1,x_3;\xi_1,\xi_3)\in T^*M_{13};\text{ there exists }
(x_2;\xi_2)\in T^*M_2,\\
&&\hspace{10ex}(x_1,x_2;\xi_1,\xi_2)\in A, 
(x_2,x_3;-\xi_2,\xi_3)\in B\}.
\ea\eneq
\end{notation}
We consider the operation of composition of kernels:
\eq\label{eq:conv00}
&&\ba{rcl}
\conv\cl\Derb(\cor_{M_{12}})\times\Derb(\cor_{M_{23}})&\to&\Derb(\cor_{M_{13}})\\
(K_1,K_2)&\mapsto&K_1\conv K_2\eqdot
\reim{q_{13}}(\opb{q_{12}}K_1\ltens\opb{q_{23}}K_2).
\ea
\eneq
Let  $A_i=\SSi(K_i)\subset T^*M_{i,i+1}$  and assume that  
\eq\label{eq:noncharker}
&&\left\{
\parbox{60ex}{
(i) $q_{13}$ is proper on $\opb{q_{12}}\supp(K_1)\cap\opb{q_{23}}\supp(K_2)$,
\\[2ex]
(ii) 
$\ba[t]{l}
\opb{p_{12}}A_1\cap\opb{p_{2^a3}}A_2
\cap (T^*_{M_1}M_1\times T^*M_2\times T^*_{M_3}M_3)\\[1ex]
\hs{10ex} \subset T^*_{M_1\times M_2\times M_3}(M_1\times M_2\times M_3).
\ea$
}\right.
\eneq
It follows from Theorem~\ref{th:opboim} that under the
assumption~\eqref{eq:noncharker} we have:
\eq\label{eq:convolution_of_kern}
&&\SSi(K_1\conv K_2)\subset A_1\aconv A_2.
\eneq

\subsubsection*{Characteristic inverse images}
Theorem~\ref{th:opboim} treats the easy cases of external tensor
product or external $\Hom$, non-characteristic inverse images or proper
direct image. 
In order to treat more general cases we introduce some  additional geometrical
notions. 

Let $\Lambda$ be a smooth Lagrangian submanifold of $T^*M$. 
The Hamiltonian isomorphism defines an isomorphism
\eqn
&&T^*\Lambda\simeq T_{\Lambda}T^*M.
\eneqn
Let $j\cl L\hookrightarrow M$ be the embedding of a smooth submanifold
$L$ of $M$. The Liouville form defines an embedding 
\eqn
&&T^*L\hookrightarrow T^*T^*_LM\simeq T_{T^*_LM}T^*M.
\eneqn
Now consider  a morphism of manifolds  $f\cl M\to N$ and let us
identify $M$ to the graph of $f$ in $M\times N$.
For a subset $B\subset T^*N$ one sets:
\eq\label{eq:sharp}
&&f^{\sharp}(B)=T^*M\cap C_{T^*_M(M\times N)}(T^*_MM\times B).
\eneq
In local symplectic coordinate systems $(x;\xi)$ on $M$ and $(y;\eta)$ on $N$ 
one has
\eq\label{def:fdiese}
&\left\{
\parbox{55ex}{
$(x_0;\xi_0)\in f^{\sharp}(B)$ if and only if there exist sequences
$\{x_n\}_n\subset M$ and  $\{(y_n;\eta_n)\}_n\subset B$ such that \\
$x_n\to x_0$, ${}^tf'(x_n)\cdot\eta_n\to[n]\xi_0$ and 
$\vert y_n-f(x_n)\vert\cdot\vert\eta_n\vert\to[n]0$.
}
\right.\eneq
For two  closed $\R^+$-conic subsets $A$ and $B$ of $T^*M$ one sets
\eq
&&A\hatp B=T^*M\cap C(A,B^a).
\eneq
Here, $C(A,B^a)$ is considered as a subset of $T^*T^*M$ via the 
Hamiltonian isomorphism and $T^*M$ is embedded into $T^*T^*M$ via the
Liouville form $\alpha_M$. 
In a local coordinate system, one has
\eq\label{def:+hat}
&\left\{
\parbox{50ex}{
$(z_0;\zeta_0)\in A\hatp B$ if and only if there exist sequences 
$\{(x_n;\xi_n)\}_n$ in $A$ and $\{(y_n;\eta_n)\}_n$ in $B$ such that
$x_n\to[n]z_0$, $y_n\to[n]z_0$, $\xi_n+\eta_n\to[n]\zeta_0$ and 
$\vert x_n-y_n\vert\cdot\vert\xi_n\vert\to[n]0$.
}
\right.\eneq

\begin{theorem}\label{th:opboim2}{\rm (See \cite[Cor.~6.4.4,~6.4.5]{KS90}.)}
Let $F_1,F_2\in\Derb(\cor_M)$ and let $G\in\Derb(\cor_N)$. Then 
\eqn
&&\SSi(F_1\ltens F_2)\subset \SSi(F_1)\hatp\SSi(F_2),\\
&&\SSi(\rhom(F_1,F_2))\subset\SSi(F_2)\hatp\SSi(F_1)^a,\\
&&\SSi(\opb{f}G)\cup\SSi(\epb{f}G)\subset f^{\sharp}(\SSi(G)).
\eneqn
\end{theorem}

\subsubsection*{Non proper direct images}
We shall also need a direct image theorem in a non proper case.

Consider a  {\em constant} linear map
$u$ of {\em trivial} vector bundles over $M$,
that is, we assume that $E_i=M\times V_i$ ($i=1,2$) and $u\cl V_1\to V_2$
is a linear map.
The map $u$ defines  the maps described by the diagram
\eqn
\xymatrix{
&T^*M\times V_1\times V^*_2\ar[ld]_-{u_d}\ar[rd]^-{u_\pi}&\\
T^*M\times V_1\times V^*_1\ar[rd]_-{v_\pi}
               &&T^*M\times V_2\times V^*_2.\ar[ld]^-{v_d}\\
&T^*M\times V_2\times V^*_1&
}\eneqn
Note that for a subset $A$ of $T^*E_1$ we have
\eq\label{eq:upiudA}
&& u_\pi(\opb{u_d}(A))= \opb{v_d}({v_\pi(A)}).
\eneq
\begin{notation}\label{not:usharp}
Let $u\cl E_1\to E_2$ be a constant linear map of trivial vector bundles
over $M$ and let $A\subset T^*E_1$ be a closed subset. We set
\eq\label{not:udag2}
&&u_{\sharp}(A)=\opb{v_d}(\ol{v_\pi(A)}).
\eneq
\end{notation}

In Lemmas~\ref{lem:indprolim} and~\ref{lem:SSreimroim} below we use the
notations $\bigoplus_nG_n$ and $\prod_nG_n$ for a family $\{G_n\}_{n\in \N}$ in
$\Derb(\cor_M)$. We define it as follows.  Let $p\cl M\times\N \to M$ be
the projection. Then we have a unique $G\in \Derb(\cor_{M\times\N})$ such that
$G|_{M\times\{n\}} \simeq G_n$, for all $n$, and we set
$\bigoplus_nG_n \eqdot \reim{p}G$ and $\prod_nG_n \eqdot \roim{p}G$.

\begin{lemma}\label{lem:indprolim}
Let $M$ be a manifolds and let $\{U_n\}_{n\in \N}$ be an increasing sequence of
open subsets of $M$ such that $M=\bigcup_n U_n$. 
Then, for any $F\in \Derb(\cor_M)$, we have the distinguished triangles
$$
 \bigoplus_n F_{U_n}  \to[\id - s_1] \bigoplus_n F_{U_n} \to F  \to[+1], 
\;\;
F \to \prod_n \rsect_{U_n}(F) \to[\id - s_2] \prod_n \rsect_{U_n}(F) \to[+1],
$$
where $s_1$ is the sum of the natural morphisms $F_{U_n} \to F_{U_{n+1}}$
and $s_2$ the product of the natural morphisms
$\rsect_{U_{n+1}}(F) \to \rsect_{U_n}(F)$ for $n\geq 0$ and the zero morphism
for $n=-1$.
\end{lemma}
\begin{proof}
These triangles arise from similar exact sequences of sheaves when
$F$ is a flabby sheaf. The exactness can be checked easily on the stalks
in the first case and on sections over any open subset in the second case.
\end{proof}

\begin{lemma}\label{lem:SSreimroim}
Let $f\cl M\to N$ be a morphism of manifolds and let $\{U_n\}_{n\in \N}$ be an
increasing sequence of open subsets of $M$ such that $M=\bigcup_n U_n$.
Then, for any $F\in \Derb(\cor_M)$, we have
$$
\SSi(\reim{f} F) \subset \ol{\bigcup_n \SSi(\reim{f} (F_{U_n}))} ,
\qquad
\SSi(\roim{f} F) \subset \ol{\bigcup_n \SSi(\roim{f} \rsect_{U_n}(F))} .
$$
\end{lemma}
\begin{proof}
We can check, similarly as in~\cite[Exe.~V.7]{KS90}, that
for any family $\{G_n\}_{n\in \N}$ in $\Derb(\cor_N)$ we have
$\SSi(\bigoplus_nG_n) \cup \SSi(\prod_nG_n) \subset \ol{\bigcup_n \SSi(G_n)}$.
Then the result follows from Lemma~\ref{lem:indprolim} and the fact that
$\reim{f}$ commutes with $\oplus$ and $\roim{f}$ with $\prod$.
\end{proof}

The following result is due to Tamarkin~\cite[Lem.~3.3]{Ta} but our proof is
completely different.
\begin{theorem}\label{th:nonpropdirim}
Let $u\cl E_1\to E_2$ be a constant linear map of trivial vector bundles
over $M$ and let
$F\in\Derb(\cor_{E_1})$. Then $\SSi(\reim{u}F)\subset u_{\sharp}(\SSi(F))$.
The same estimate holds with $\reim{u}F$ replaced with $\roim{u}F$.
\end{theorem}
\begin{proof}
(i) By decomposing $u$ by its graph, one is reduced to prove the result for an
immersion and for a projection. Since the case of an immersion is obvious, we
restrict ourselves to the case where $E=M\times V$ and $u\cl E\to M$ is the
projection. Moreover the result is local on $M$ and we may assume that $M$ is 
an open subset in a vector space $W$.

\medskip\noindent
(ii) We consider $(x_0;\xi_0) \in T^*M\simeq M\times W^*$ such that 
$(x_0;\xi_0) \not\in u_{\sharp}(\SSi(F))$. We will prove that
$(x_0;\xi_0) \not\in \SSi(\reim{u}F)\cup \SSi(\roim{u}F)$.
If $\xi_0=0$, then $F|_{U\times V}\simeq 0$ for some neighborhood $U$ of $x_0$
and the result follows easily. Hence we assume that $\xi_0\not=0$.
Up to shrinking $M$ we may find an open cone $C\subset W^*\times V^*$ such that
$(\xi_0,0) \in C$ and $\SSi(F) \cap ((M\times V)\times C)=\emptyset$.

\medskip\noindent
(iii) We choose an open convex cone $\gamma\subset W\times V$ such that
$\ol{\gamma} \cap (\{0\}\times V) = \{(0,0)\}$ and $\gamma^\circ \subset C$.
We also choose two sequences of points $\{z_n\}_{n\in\N}$,
resp. $\{z'_n\}_{n\in\N}$,  of $W\times V$ such that $W\times V$ is the
increasing union of the cones $\gamma_n = z_n-\gamma$, resp.
$\gamma'_n = z'_n+\gamma$.
By Lemma~\ref{lem:SSreimroim} it is enough to show
$$
(\SSi(\roim{u}\rsect_{\gamma_n}F)  \cup \SSi(\reim{u}(F_{\gamma'_n})))
\cap (M\times (C\cap (W^*\times\{0\}))  =\emptyset.
$$

\medskip\noindent
(iv) By Lemma~\ref{le:sscone} below
$\SSi(\cor_{\ol{\gamma_n}}) \subset (W\times V) \times (-C)$.
Using $\RD'_M(\cor_{\ol{\gamma_n}}) \simeq \cor_{\gamma_n}$ we deduce
$\SSi(\cor_{\gamma_n})\subset (W\times V) \times C$.
Similarly $\SSi(\cor_{\gamma'_n})\subset (W\times V) \times (-C)$.
Since $\SSi(F) \cap ((M\times V)\times C)=\emptyset$,
Corollary~\ref{cor:opboim} gives
$$
(\SSi(\rsect_{\gamma_n}F)  \cup \SSi((F_{\gamma'_n})))
\cap ((M\times V)\times C)  =\emptyset.
$$
Since $\ol{\gamma} \cap (\{0\}\times V) = \{(0,0)\}$ the map 
$u\cl M\times V \to M$ is proper on all $\ol{\gamma_n}$ and $\ol{\gamma'_n}$
and the result follows from Theorem~\ref{th:opboim}~(ii).
\end{proof}

For a trivial vector bundle $E=M\times V$ we denote by
\eq\label{eq:projfibre}
\piw_E\cl T^*E \to T^*M\times V^*,
\eneq
or $\piw$ if there is no risk of confusion, the natural projection.
We say that a subset of $T^*M \times V^*$ is a cone if it is stable by
the multiplicative action of $\R^+$ given by
\eq\label{eq:defcone}
\lambda\cdot (x;\xi,v) = (x;\lambda\xi, \lambda v).
\eneq
We will be mainly concerned with the case where $F\in\Derb(\cor_E)$
has a microsupport bounded by $\opb{\piw_E}(A)$ for some closed cone
$A\subset T^*M\times V^*$.

Let $u\cl E_1=M\times V_1\to E_2=M\times V_2$ be a constant linear map of 
trivial vector bundles over $M$ and denote by 
\eq\label{eq:utildd}
&&\tw u_d\cl T^*M\times V^*_2\to T^*M\times V^*_1
\eneq
the map associated with $u$. 

\begin{corollary}\label{cor:easyestimatproperm}
Let $u\cl E_1\to E_2$ be a constant linear map of trivial vector bundles
over $M$ and let $F\in\Derb(\cor_{E_1})$. Assume that 
$\SSi(F)\subset \opb{\piw_{E_1}}(A_1)$ for a closed cone $A_1\subset T^*M\times V^*_1$.
Then $\SSi(\reim{u}F)\subset \opb{\piw_{E_2}}\opb{\tw u_d}(A_1)$.
The same estimate holds with $\reim{u}F$ replaced with $\roim{u}F$.
\end{corollary}
\begin{proof}
We have $v_\pi(\opb{\piw_{E_1}}(A_1)) = A_1 \times V_2$ and this set is closed. We thus  have 
\eqn
u_{\sharp}(\opb{\piw_{E_1}}(A_1)) &=&\opb{v_d}({v_\pi(\opb{\piw_{E_1}}(A_1))})
=u_\pi(\opb{u_d}(A_1\times V_1))\\
&=&\opb{\tw u_d}(A_1)\times V_2=\opb{\piw_{E_2}}\opb{\tw u_d}(A_1).
\eneqn
\end{proof}

\subsubsection*{Localization}
Let $\sht$ be a triangulated category, $\shn$ a null system, that is, a
full triangulated
subcategory with the property that if one has an isomorphism 
$F\simeq G$ in $\sht$ with $F\in\shn$, then $G\in\shn$.
The localization $\sht/\shn$ is a well defined triangulated category
(we skip the problem of universes). Its objects are those of $\sht$
and a morphism $u\cl F_1\to F_2$ in $\sht$ becomes an isomorphism in 
$\sht/\shn$ if, after embedding  this morphism in a distinguished
triangle $F_1\to F_2\to F_3\to[+1]$, one has $F_3\in\shn$. 

Recall that the left orthogonal $\shn^{\perp,l}$ of
$\shn$ is the full triangulated subcategory of $\sht$ defined by:
\eqn
&&\shn^{\perp,l}=\{F\in\sht;\Hom[\sht](F,G)\simeq0\text{ for all }G\in\shn\}.
\eneqn
By classical results (see {\em e.g.,}~\cite[Exe.~10.15]{KS05}), if 
the embedding $\shn^{\perp,l}\hookrightarrow\sht$ admits a left adjoint, or 
equivalently, if for any $F\in\sht$, there exists a
distinguished triangle $F'\to F\to F''\to[+1]$ with 
$F'\in\shn^{\perp,l}$ and $F''\in\shn$, then there is an equivalence 
$\shn^{\perp,l}\simeq\sht/\shn$.

Of course, there are similar results with the right orthogonal $\shn^{\perp,r}$.

Now let $U$ be a subset of $T^*M$ and set $Z=T^*M\setminus U$. The full
subcategory $\Derb_Z(\cor_M)$ of $\Derb(\cor_M)$ consisting of sheaves $F$
such that $\SSi(F)\subset Z$ is a null system. One sets
\eqn
&& \Derb(\cor_M;U)\eqdot \Derb(\cor_M)/\Derb_Z(\cor_M),
\eneqn
the localization of $\Derb(\cor_M)$ by $\Derb_Z(\cor_M)$. Hence, the
objects of $\Derb(\cor_M;U)$ are those of $\Derb(\cor_M)$ but a
morphism $u\cl F_1\to F_2$ in $\Derb(\cor_M)$ becomes an isomorphism
in $\Derb(\cor_M;U)$ if, after embedding  this morphism in a distinguished
triangle $F_1\to F_2\to F_3\to[+1]$, one has $\SSi(F_3)\cap U=\emptyset$. 

For  a closed subset $A$ of $U$, $\Derb_A(\cor_M;U)$ denotes  
the full triangulated subcategory of $\Derb(\cor_M;U)$ consisting of
objects whose microsupports have an intersection with $U$ contained in $A$. 

\subsubsection*{Quantized symplectic isomorphisms  (\cite[\S 7.2]{KS90})}
Consider two manifolds $M$ and $N$, two conic open subsets $U\subset T^*M$
and $V\subset T^*N$ and a homogeneous symplectic isomorphism $\chi$:
\eq\label{eq:contact1}
&& T^*N\supset V\isoto[\chi] U\subset T^*M.
\eneq
Denote by $V^a$ the image of $V$ by the antipodal map $a_N$ on $T^*N$ and by $\Lambda$
the image of the graph of $\phi$ by $\id_U \times a_N$. Hence
$\Lambda$ is a conic Lagrangian submanifold of $U\times V^a$.
A quantized contact transformation (a QCT, for short) above $\chi$ is a kernel
$K\in\Derb(\cor_{M\times N})$ such that 
$\SSi(K)\cap(U\times V^a)\subset\Lambda$ and satisfying some technical
properties that we do not recall here, so that the kernel $K$ induces
an equivalence of categories 
\eq\label{eq:QCT1}
&&K\conv\scbul\cl\Derb(\cor_N;V)\isoto\Derb(\cor_M;U). 
\eneq
Given $\chi$ and $q\in V$,
$p=\chi(q)\in U$, there exists such a QCT after replacing $U$ and $V$
by sufficiently small neighborhoods of $p$ and $q$.

\subsubsection*{Simple sheaves (\cite[\S 7.5]{KS90})}
Let $\Lambda\subset\dTM$ be a locally closed conic Lagrangian
submanifold and let $p\in\Lambda$.
Simple sheaves along $\Lambda$ at $p$ 
are defined in~\cite[Def.~7.5.4]{KS90}. 

When $\Lambda$ is the conormal bundle to a submanifold
$N\subset M$, that is, when the projection $\pi_M\vert_\Lambda\cl\Lambda\to M$ has
constant rank, then an object 
$F\in\Derb(\cor_M)$ is simple along $\Lambda$ at
$p$ if $F\simeq\cor_N\,[d]$ in $\Derb(\cor_M;p)$ for some shift $d\in\Z$.

If $\SSi(F)$ is contained in $\Lambda$ on a neighborhood of $\Lambda$,
$\Lambda$ is connected and $F$ is simple at some point of $\Lambda$,
then $F$ is simple at every point of $\Lambda$.

\subsubsection*{The functor $\muhom$ (\cite[\S 4.4, \S 7.2]{KS90})}
The functor of microlocalization along a submanifold has been introduced by
Mikio Sato in the 70's and has been at the origin of what is now called
``microlocal analysis''.  A variant of this functor, the bifunctor
\eq
&&\muhom\cl\Derb(\cor_M)^\rop\times\Derb(\cor_M)\to\Derb(\cor_{T^*M})
\label{eq:muhom}
\eneq
has been constructed in~\cite{KS90}.
Let us only recall the properties of this functor that we shall use.
For $F,G\in \Derb(\cor_M)$, with $F$ cohomologically constructible,
we have
\eqn
\roim{\pi_M} \muhom(F,G) &\simeq& \rhom(F,G) , \\
\reim{\pi_M} \muhom(F,G) &\simeq& \RD'_M(F) \ltens G 
\eneqn
and we deduce the distinguished triangle
\eq
\label{eq:dtmuhom}
&& \RD'_M(F) \ltens G \to \rhom(F,G)
\to \roim{\dot\pi_M{}} ( \muhom(F,G) |_{\dTM}) \to[+1].
\eneq
Let $\Lambda\subset\dTM$ be a locally closed smooth conic Lagrangian
submanifold
and let $F\in\Derb(\cor_M)$ be simple along $\Lambda$. Then
\eq
\label{eq:simplemuhom}
&& \muhom(F,F) |_{\Lambda} \simeq \cor_\Lambda.
\eneq

\section{Quantization of Hamiltonian isotopies}
\label{section:GKS} 
In this section, we recall the main theorem of~\cite{GKS10}.

We first recall some notions of symplectic geometry. Let $\symx$ be a
symplectic manifold with symplectic form $\omega$. We denote by $\symx^a$ the
same manifold endowed with the symplectic form $-\omega$.  The symplectic
structure induces the Hamiltonian isomorphism $\h\cl T\symx \isoto T^*\symx$
by $\h(v) = \iota_v(\omega)$, where $\iota_v$ denotes the
contraction with $v$
(in case $\symx$ is a cotangent bundle we have $\h=-H^{-1}$, where $H$ is used
in Definition~\ref{def:coisotropic}).
To a vector field $v$ on $\symx$ we associate in this way a $1$-form
$\h(v)$ on $\symx$.  For a $C^\infty$-function
$f\cl \symx\to \R$, the Hamiltonian vector field of $f$ is by definition
$H_f \eqdot -\h^{-1}(df)$.

A vector field $v$ is called symplectic if its flow preserves $\omega$.
This is equivalent to $\shl_v(\omega) = 0$ where $\shl_v$ denotes
the Lie derivative of $v$.  By Cartan's formula
($\shl_v= d\,\iota_v+ \iota_v\,d$) this is again equivalent
to $d(\h(v)) = 0$ (recall that $d\omega=0$).  The vector field $v$ is
called Hamiltonian if $\h(v)$ is exact, or equivalently
$v=H_f$ for some function $f$ on $\symx$.

Let $I$ be an open interval of $\R$ containing the origin and let
$\Phi\cl \symx \times I\to \symx$ be a map such that
$\phi_s\eqdot\Phi(\cdot,s)\cl \symx \to \symx$ is a symplectic isomorphism for
each $s\in I$ and is the identity for $s=0$. The map $\Phi$ induces a time
dependent vector field on $\symx$
\eq
&& v_\Phi \eqdot \frac{\partial\Phi}{\partial s} \cl \symx\times I\to T\symx.
\eneq
The ``time dependent'' $1$-form
$\beta = \h(v_\Phi) \cl \symx\times I\to T^*\symx$
satisfies $d(\beta_s) =0$ for any $s\in I$.
The map $\Phi$ is called a Hamiltonian isotopy if $v_{\Phi,s}$
is Hamiltonian, that is, if $\beta_s$ is exact, for any $s$.  In
this case we can write $\beta_s = - d(f_s)$ for some $C^\infty$-function 
$f\cl\symx \times I\to\R$. 
Hence we have
\eqn
&&\frac{\partial\Phi}{\partial s} =H_{f_s}.
\eneqn

\medskip

The fact that the isotopy $\Phi$ is Hamiltonian can be interpreted as a
geometric property of its graph as follows.  For a given $s\in I$ we let
$\Lambda_{s}$ be the graph of $\phi_s^{-1}$
and we let $\Lambda'$ be the family of $\Lambda_s$'s:
\eqn
&&\Lambda_s=\set{(\phi_s(v),v)}{v\in\symx^a} \subset \symx\times\symx^a, \\
&&\Lambda'=
\set{(\phi_s(v),v,s)}{v\in\symx^a,\;s\in I}\subset\symx\times\symx^a\times I.
\eneqn
Thus $\Lambda_s$ is a Lagrangian submanifold of $\symx \times \symx^a$. Now we
can see that $\Phi$ is a Hamiltonian isotopy if and only if
there exists a Lagrangian submanifold  $\Lambda \subset \symx \times \symx^a \times T^*I$
such that, for any $s\in I$,
\eq\label{eq:lambdacirclambda}
&&\Lambda_{s} = \Lambda\conv T^*_{s}I.
\eneq
(Here, the notation $\scbul\conv\scbul$ is a slight generalization
of~\eqref{eq:convolution_of_sets} to the case where the symplectic manifolds are no
more cotangent bundles.)
In this case  $\Lambda$ is written
\eq
\label{eq:def-lambda}
\Lambda & = &
\set{\bl\Phi(v,s), v, s, -f(\Phi(v,s),s)\br}{ v\in\symx, s\in I},
\eneq
where the function $f\cl \symx \times I\to\R$ is defined up to addition of a
function depending on $s$ by $v_{\Phi,s} = H_{f_s}$.

\subsubsection*{Homogeneous case}
Let us come back to the case $\symx=\dTM$ and  
consider $\Phi\cl \dTM\times I\to \dTM$ such that 
\eq\label{hyp:isot1}
&&\begin{cases}
\mbox{$\phi_s$ is a homogeneous symplectic isomorphism for each $s\in I$,} \\
\phi_0 = \id_{\dTM}.
\end{cases}
\eneq
In this case $\Phi$ is a Hamiltonian isotopy and there exists a unique
homogeneous function $f$ such that $v_{\Phi,s} = H_{f_s}$. It is given by
\eq\label{eq:defin_of_f}
f = \langle \alpha, v_\Phi \rangle 
\cl \dTM\times I\to \R.
\eneq
Since $f$ is homogeneous of degree $1$ in the fibers of $\dTM$, 
the Lagrangian submanifold $\Lambda$ of
$\dTM\times \dTM\times T^*I$ associated to $f$ in~\eqref{eq:def-lambda}
is $\R^+$-conic.

We say that $F\in \Der(\cor_M)$ is locally bounded if for any relatively
compact open subset $U\subset M$ we have $F|_U \in \Derb(\cor_U)$.  We denote
by $\Derlb(\cor_M)$ the full subcategory of $\Der(\cor_M)$ consisting of
locally bounded objects.
\begin{theorem}\label{th:3}{\rm(\cite[Th~4.3]{GKS10}.)}
Consider a  homogeneous Hamiltonian isotopy $\Phi$ 
satisfying the hypotheses~\eqref{hyp:isot1}.
Let us consider the following
conditions on $K\in\Derlb(\cor_{M\times M\times I})$:
\banum
\item
$\SSi(K)\subset\Lambda\cup T^*_{M\times M\times I}(M\times M\times I)$,
\label{cond:qhi1}
\item
$K_0\simeq \cor_\Delta$,
\label{cond:qhi3}
\item
both projections $\Supp(K)\tto M\times I$ are proper,
\label{cond:qhi4}
\item
$K_s\conv\opb{K_s}\simeq \opb{K_s}\conv K_s\simeq \cor_\Delta$,
where $K_s^{-1}=v^{-1}\rhom(K_s,\omega_M\etens\cor_M)$ and $v(x,y)=(y,x)$.
\label{cond:qhi2}
\eanum
Then we have
\bnum
\item
The conditions~\eqref{cond:qhi1} and~\eqref{cond:qhi3}
imply the other two conditions~\eqref{cond:qhi4} and~\eqref{cond:qhi2}.
\item
There exists $K$ satisfying \eqref{cond:qhi1}--\eqref{cond:qhi2}.
\item
Moreover
such a $K$ satisfying the conditions \eqref{cond:qhi1}--\eqref{cond:qhi2}
is unique up to a unique isomorphism.
\enum
\end{theorem}
We shall call $K$ the {\em quantization} of $\Phi$ on $I$, or the quantization
of the family $\{\phi_s\}_{s\in I}$.

\subsubsection*{Non homogeneous case}
Theorem~\ref{th:3} is concerned with homogeneous Hamiltonian isotopies.  The
next result will allow us to adapt it to non homogeneous cases.
Let $\Phi\cl T^*M\times I\to T^*M$ be a Hamiltonian isotopy and assume
\eq\label{eq:hyp-support-isot}
\left\{\begin{minipage}[c]{9cm}
there exists a compact set $C\subset T^*M$ such that
$\phi_s|_{T^*M \setminus C}$ is the identity for all $s\in I$.
\end{minipage}\right.
\eneq
We denote by  $T^*_{\{\tau>0\}}(M\times\R)$ the open subset $\{\tau>0\}$ of
$T^*(M\times\R)$ and we define the map 
\eq\label{eq:rho}
&&\rho\cl T^*_{\{\tau>0\}}(M\times\R)\to T^*M,\quad (x,t;\xi,\tau)\mapsto
(x;\xi/\tau).
\eneq
\begin{proposition}\label{pro:homnonhomHIso}{\rm(\cite[Prop.~A.6]{GKS10}.)}
There exist a homogeneous Hamiltonian isotopy
$\tw \Phi\cl \dT^*(M\times \R)\times I\to \dT^*(M\times\R)$
and $C^\infty$-functions $u\cl T^*M \times I\to\R$  and $v\cl I\to\R$
such that the following diagram commutes:
\eqn
&&\xymatrix@C=2cm{
T^*_{\{\tau>0\}}(M\times \R)\times I\ar[r]^{\tw\Phi}\ar[d]_{\rho\times\id_I}
                           & T^*_{\{\tau>0\}}(M\times \R)\ar[d]_\rho  \\
T^*M\times I\ar[r]^{\Phi}              & T^*M 
}\eneqn
and
\eq\label{eq:twphi}
&&\tw \Phi((x;\xi),(t;\tau),s) = ((x';\xi'), (t+ u(x;\xi/\tau,s);\tau)) , \\
\label{eq:twphi2}
&&\tw \Phi((x;\xi),(t;0),s) = ((x;\xi), (t+ v(s);0)),
\eneq
where $(x';\xi'/\tau) = \phi_s(x;\xi/\tau)$.
Moreover we have $u(x;\xi/\tau,s) = v(s)$ for $(x;\xi/\tau)\not\in C$.
\end{proposition}

\section{Convolution and localization}\label{section:conv}
Most of the ideas of this section are due to Tamarkin~\cite{Ta}.
The reader will be aware that our notations do not follow Tamarkin's
ones. We also give some proofs
which may be rather different from Tamarkin's original ones.

In all this section, we consider a trivial vector bundle
\eq
&&q\cl E=M\times V \to M
\eneq
and a trivial cone $\gamma= M\times\gamma_0\subset E$ such that
\eq\label{eq:hypcone}
&& \mbox{$\gamma_0$ is a closed convex proper cone of $V$ containing 
$0$ and $\gamma_0\not=\{0\}$.}
\eneq
The polar cone $\gamma_0^\circ\subset V^*$ is 
 the closed convex cone given by 
\eqn
&&\gamma_0^\circ= \{\theta\in V^*; \langle\theta,v\rangle \geq 0\}
\mbox{ for all }v\in\gamma_0.
\eneqn
Many results could be generalized to general
vector bundles and general proper convex cones, but in practice we 
shall use these results with $V=\R$ and $\gamma_0=\{t\in\R;t\geq0\}$.  
Recall that a subset in $T^*M\times V^*$ is a cone if it is invariant 
by the diagonal action of $\R^+$ (see~\eqref{eq:defcone}).

\begin{definition}\label{def:strictcone}
A closed cone $A\subset T^*M\times V^*$ is called a strict $\gamma$-cone
if $A\subset (T^*M\times \Int\gamma^\circ_0)\cup T^*_MM\times \{0\}$.
\end{definition}
\begin{example}\label{exa:strictgc}
Assume $V=\R$ and $M$ is open in $\R^n$.
Denote by $(t;\tau)$ the coordinates on $T^*\R$ and by 
$(x;\xi)$ the coordinates on $T^*M$. Let 
$\gamma_0=\{t\in\R;t\geq0\}$.  Then a closed cone $A\subset T^*M\times V^*$
is a strict $\gamma$-cone if, for any compact subset $C\subset M$,
there exists $a\in\R, a>0$ such that  $\tau\geq a\vert\xi\vert$ for all
$(x;\xi,\tau)\in A\cap (\opb{\pi_M}(C)\times V^*)$.
\end{example} 
\begin{remark}\label{rem:strictgc}
  If $f\cl N\to M$ is a morphism of manifolds and $A\subset T^*M\times V^*$
  is a strict $\gamma$-cone, then $f\times \id_V\cl N\times V \to M\times V$
  is non-characteristic for $\opb{\piw_E}(A)$ (where $\opb{\piw_E}$ is
  defined in~\eqref{eq:projfibre}).
\end{remark}

In the sequel, we consider the maps
\eq\label{eq:s}
&\ba{rcl}
&q_1,q_2,s\cl V\times V\to V,\\
&q_1(v_1,v_2)=v_1,\quad q_2(v_1,v_2)=v_2,\quad s(v_1,v_2)= v_1+v_2.
\ea\eneq
If there is no risk of confusion, we still denote by $q_1,q_2,s$ the associated maps 
$M\times V\times V\to M\times V$. 

We denote by $\delta_M$  the diagonal embedding
\eq\label{eq:delta}
&&\delta_M\cl M\hookrightarrow M\times M
\eneq
 and if there is no risk of confusion, 
we still denote by $\delta_M$ the associated map 
$M\times V\times V\hookrightarrow M\times M\times V\times V$,
that is, the map $E\times_ME\hookrightarrow E\times E$.

The maps $s$ and $\delta_M$ give rise to  the maps:
\eqn
&&T^*(E\times_ME)\from[(\delta_M)_d]M\times_{M\times M}T^*(E\times_M E)
\to[(\delta_M)_\pi]T^*(E\times E),\\
&&T^*(E\times_{M}E)\from[s_d]V\times_{V\times V}T^*(E\times_M E)\to[s_\pi]T^*E.
\eneqn
On $T^*E$ we have the antipodal map $a$, but there is another involution
associated with $a$ and the involution $(x,y)\mapsto (x,-y)$ on $E$.
We denote by $\alpha$ the involution of  $T^*E$
\eq\label{eq:involalpha}
&&\alpha\cl (x,y;\xi,\eta)\mapsto(x,-y;-\xi,\eta)
\eneq
and for a subset $A\subset T^*E$ we denote by $A^\alpha$ its image by
this involution.
We also denote by $\alpha$ the involution of $T^*M\times V^*$ defined by
$(x;\xi,\eta)\mapsto(x;-\xi,\eta)$. Hence for $A \subset T^*M\times V^*$
we have, using the notation~\eqref{eq:projfibre},
$\opb{\piw_E}(A^\alpha) = \opb{\piw_E}(A)^\alpha$.

\subsubsection*{Convolution}
Recall the notations~\eqref{eq:sharp} and~\eqref{not:udag2}.
\begin{notation}\label{not:starhat}
For two closed subsets $A$ and $B$ in $T^*E$,
we set
\eq\label{eq:convsets}
&&A\hatstar B\eqdot s_\sharp \delta_M^{\sharp}(A\times B).
\eneq
\end{notation}

In general, the calculation of $A\hatstar B$ is difficult. 
In Lemmas~\ref{le:boundforhatstar} and~\ref{lem:star-gamma-cones} 
below we consider special situations in which this calculation is easy.
\begin{lemma}\label{le:boundforhatstar}
Let $A'$ and $B'$ be two closed cones in $V^*$. Set 
$A = T^*M\times V\times A'$ and $B=T^*M\times V \times B'$.
Then
\eq\label{eq:easy-bound-hatstar}
&&A\hatstar B = A\cap B.
\eneq
\end{lemma}
\begin{proof}
Using the hypothesis on $A$ and $B$, it follows from~\eqref{def:fdiese} that
\eqn
&&\delta_M^{\sharp}(A\times B)=T^*M\times V\times V\times A'\times B'.
\eneqn
Then the result follows from Corollary~\ref{cor:easyestimatproperm}.
\end{proof}

\begin{notation}\label{not:+hatM}
  Let $A$ and $B$ be two closed cones in $T^*M\times V^*$. We set
\eq\label{eq:+hatM}
&&\ba{rcl}&&A\hplus[M]B=\{ (x;\xi,\eta) \in T^*M \times V^*;
\mbox{there exist $\xi_1,\xi_2 \in T^*_xM$ such} \\
&& \hspace{2.5cm}
\mbox{that $(x;\xi_1,\eta)\in A$, $(x;\xi_2,\eta)\in B$
and $\xi=\xi_1+\xi_2$}\}.\ea
\eneq
\end{notation}

\begin{lemma}\label{lem:star-gamma-cones}
Consider two closed strict $\gamma$-cones $A$ and $B$ in $T^*M \times V^*$.
Then $A\hplus[M]B$ is also a strict $\gamma$-cone and
$\opb{\piw}_E(A)\hatstar \opb{\piw}_E(B) = \opb{\piw}_E(A\hplus[M]B)$.

In particular, if $A\cap B\subset T^*_MM\times\{0\}$, then
\eqn
&&(\opb{\piw}_E(A)\hatstar 
(\opb{\piw}_E(B))^\alpha)\cap (T^*_MM\times T^*V)\subset T^*_EE.
\eneqn
\end{lemma}
\begin{proof}
The fact that $A\hplus[M]B$ is a strict $\gamma$-cone follows
easily from the definition.

By Remark~\ref{rem:strictgc},
$\opb{\piw}_E(A)\times \opb{\piw}_E(B)$ is non-charac\-te\-ris\-tic for the
inclusion $\delta_M\cl M \times V \times V \to M\times M \times V \times V$
and we may replace $\delta_M^{\sharp}$ by
$\delta_{M,d} \opb{\delta_{M,\pi}}$ in~\eqref{eq:convsets}.
We find $\delta_M^{\sharp}(\opb{\piw}_E(A)\times (\opb{\piw}_E(B)))
= \opb{\piw}_{M\times V \times V}(C_1)$, where
\eqn
&& C_1 = \{ (x;\xi,\eta_1,\eta_2) \in T^*M \times V^*\times V^*;
\mbox{there exist $\xi_1,\xi_2 \in T^*_xM$ such} \\
&&  \hspace{3.5cm}
\mbox{that $(x;\xi_1,\eta_1) \in A$, $(x;\xi_2,\eta_2) \in B$
and $\xi=\xi_1+\xi_2$}\}
\eneqn
and the result follows.
\end{proof}
Using the notations~\eqref{eq:s}, the convolution of sheaves is defined by:
\begin{definition}\label{def:convs}
For $F,G\in\Derb(\cor_{E})$, we set
\eq
F\star G&\eqdot&
\reim{s}(\opb{q_1}F\ltens\opb{q_2}G)\simeq\reim{s}\opb{\delta_M}(F\letens G),
\label{eq:convtens}\\
F\npstar G&\eqdot& 
\roim{s}(\opb{q_1}F\ltens\opb{q_2}G)\simeq\roim{s}\opb{\delta_M}(F\letens G).
\label{eq:NPconvtens}
\eneq
\end{definition}
The morphism $\cor_\gamma\to\cor_{M\times \{0\}}$ gives the morphism
\eq\label{eq:convgamma-id}
&&F\npstar\cor_\gamma\to F.
\eneq
Recall the following result:
\begin{proposition}{\rm (Microlocal cut-off lemma~\cite[Prop.~5.2.3,~3.5.4]{KS90})}
\label{prop:cut-offlemma}
Let $F\in\Derb(\cor_{E})$. Then
$\SSi(F) \subset T^*M \times V\times\gamma_0^\circ$
if and only if the morphism~\eqref{eq:convgamma-id} is an isomorphism.
\end{proposition}
If $\gamma_0$ has a non-empty interior we have
$\cor_{\gamma_0} \simeq \RD'_V(\cor_{\Int \gamma_0})$ and we deduce from
Corollary~\ref{cor:opboim} (ii) that
\eq\label{eq:convgammabis}
F\npstar \cor_{\gamma}
\simeq\roim{s} \rsect_{M\times V \times \Int \gamma_0} (\opb{q_1}F).
\eneq

Following Tamarkin~\cite{Ta}, we introduce a right adjoint to the
convolution functor by setting for $F,G\in\Derb(\cor_{E})$ 
\eq\label{eq:convhom}
\rhomc(G,F)&\eqdot&\roim{q_1}\rhom(\opb{q_2}G,\epb{s}F).
\eneq
Hence for $F_1,F_2,F_3\in\Derb(\cor_{E})$, we have
\eq\label{eq:adjstarhomc}
&&\RHom(F_1\star F_2,F_3)\simeq\RHom(F_1,\rhomc(F_2,F_3)).
\eneq
We use the notation:
\eq\label{not:i}
&&i\cl E\to E \mbox{ denotes the involution $(x,y)\mapsto (x,-y)$.}
\eneq
\begin{lemma}\label{lem:convhom2}
For $F,G\in\Derb(\cor_{E})$ we have
\eqn
\rhomc(G,F) &\simeq& \roim{s}\rhom(\opb{q_2}\opb{i}G, \epb{q_1}F), \\
F \star G &\simeq& \reim{q_1}( \opb{s} F \ltens \opb{q_2}\opb{i} G) .
\eneqn
\end{lemma}
\begin{proof}
We only prove the first isomorphism, the second one being similar.
We set $f\eqdot (s,-q_2)\cl E\times_M E \to E\times_M E$, $(x,v_1,v_2)
\mapsto (x, v_1+v_2,-v_2)$.
We find $f\circ f =\id$, $s = q_1 \circ f$, $q_2 \circ f = i\circ q_2$.
Since $f$ is an isomorphism $\rhom$ commutes with $\opb{f} \simeq \epb{f}$.
Since $f \circ f =\id$ we have $\opb{f} = f_*$. We deduce the isomorphisms:
\eqn
\begin{aligned}
 \rhomc(G,F)
&\simeq \roim{q_1}\rhom(\opb{q_2}G,\epb{s}F)  \\
&\simeq \roim{q_1}\rhom(\opb{f}\opb{q_2}\opb{i}G, \epb{f}\epb{q_1}F)  \\
&\simeq \roim{q_1}\opb{f}\rhom(\opb{q_2}\opb{i}G, \epb{q_1}F)  \\
&\simeq \roim{s}\rhom(\opb{q_2}\opb{i}G, \epb{q_1}F).
\end{aligned}
\eneqn
\end{proof}
\begin{proposition}\label{pro:convconv}
For $F_1,F_2,F_3\in\Derb(\cor_{E})$ we have
\eq\label{eq:convconv1}
&\ba{rcl}
(F_1\star F_2)\star F_3 &\simeq& F_1\star(F_2\star F_3),\\
\rhomc(F_1\star F_2,F_3) &\simeq& \rhomc(F_1,\rhomc(F_2,F_3)).
\ea
\eneq
\end{proposition}
\begin{proof}
(i) The first isomorphism is proved in the same way  as the associativity of the
composition of kernels: we check easily that both sides are isomorphic to
$\reim{\sigma}(\opb{q_1}(F_1) \ltens \opb{q_2}(F_2) \ltens \opb{q_3}(F_3))$
where $\sigma \cl M\times V^3 \to M\times V$ is given by
$\sigma(x,v_1,v_2,v_3) = (x,v_1+v_2+v_3)$ and $q_i\cl M\times V^3 \to M\times V$
is the projection on the $i^{th}$ factor $V$.

\medskip\noindent
(ii) We use the Yoneda embedding to prove the second isomorphism.
We apply the functor
$\Hom[\Derb(\cor_E)](H,\scbul)$ for any $H\in\Derb(\cor_E)$ to each term of
this formula. One gets an isomorphism in view of the adjunction
isomorphism~\eqref{eq:adjstarhomc} and the associativity of $\star$
proved in~(i).
\end{proof}

\begin{proposition}\label{prop:hom-homc}
Let $q\cl E\to M$ and $q'\cl M\times V \times V \to M$ be the projections.
For $F,G,H \in\Derb(\cor_{E})$ we have
\eq
\label{eq:hom-homc}
&\roim{q}(\rhom(F,\rhomc(G,H))) \simeq \roim{q} (\rhom(F\star G,H)) , \\
\label{eq:tens-star}
&\ba{rcl}
\reim{q}((F\star G) \ltens H)
&\simeq& \reim{q'}( \opb{q_1} F \ltens \opb{q_2} G  \ltens \opb{s} H) \\
&\simeq& \reim{q}(F \ltens  (\opb{i}G\star H)).
\ea
\eneq
\end{proposition}
\begin{proof}
The first isomorphism follows by adjunction from~\eqref{eq:convtens}
and~\eqref{eq:convhom}, using $q\circ q_1 = q \circ s$.
The second and third ones follow from the projection formula, the identities
$q\circ q_1 = q' = q \circ s$ and Lemma~\ref{lem:convhom2}.
\end{proof}
Recall that the involution $(\scbul)^\alpha$ is defined
in~\eqref{eq:involalpha}.
\begin{proposition}\label{pro:SSconv}
For $F,G\in\Derb(\cor_{E})$ we have
\eq\label{eq:SSstar1}
&&\ba{rcl}
&&\SSi(F\star G)\subset\SSi(F)\hatstar \SSi(G),\\
&&\SSi(\rhomc(G,F))\subset \SSi(F)\hatstar \SSi(G)^\alpha .
\ea
\eneq
\end{proposition}
\begin{proof}
Both inclusions in~\eqref{eq:SSstar1} follow from~\eqref{eq:convtens},
\eqref{eq:convsets} and Theorems~\ref{th:opboim2} and~\ref{th:nonpropdirim}.
For the second one we also use Lemma~\ref{lem:convhom2}
and $\SSi(\opb{i}G)^a = \SSi(G)^\alpha$.
\end{proof}
Using~\eqref{eq:SSstar1} and~\eqref{eq:easy-bound-hatstar}, we get:
\begin{corollary}
Let $F,G\in\Derb(\cor_{E})$ and assume that 
there exist closed cones $A', B'\subset V^*$
such that $\SSi(F) \subset T^*M \times V \times A'$
and $\SSi(G) \subset T^*M \times V \times B'$. Then
\eq\label{eq:SSstar1b}
&&\ba{rcl}
&&\SSi(F\star G)\subset   T^*M \times V \times (A' \cap B'), \\
&&\SSi(\rhomc(G,F))\subset  T^*M \times V \times (A' \cap B').
\ea
\eneq
\end{corollary}

\begin{corollary}\label{cor:restr-rhomc}
Let $F,G\in\Derb(\cor_{E})$ and assume that there exist closed strict
$\gamma$-cones $A$ and $B$ in $T^*M \times V^*$ such that 
$\SSi(F) \subset \opb{\piw}_E(A)$ and $\SSi(G) \subset \opb{\piw}_E(B)$.
Let $N$ be a submanifold of $M$ and $j\cl N\times V \to M\times V$ the
inclusion. Then
$$
\opb{j}\rhomc(F,G) \simeq \rhomc(\opb{j}F, \opb{j}G).
$$
\end{corollary}
\begin{proof}
By Proposition~\ref{pro:SSconv} and Lemma~\ref{lem:star-gamma-cones},
$\SSi(\rhomc(F,G)) \subset \opb{\piw}_E(A\hplus[M]B)$ and $A\hplus[M]B$ is a
strict $\gamma$-cone.
By Remark~\ref{rem:strictgc}, we deduce
$\epb{j}H \simeq \opb{j}H \otimes \omega_{N\times V|M\times V}$
for $H=F,G$ or $\rhomc(F,G)$.
This gives the first and last steps in the sequence of isomorphisms,
where we set  $j'=j\times \id_V$:
\begin{align*}
\opb{j}\rhomc(F,G)
&\simeq \epb{j}\roim{q_1}\rhom(\opb{q_2}F,\epb{s}G)
  \otimes \omega_{N\times V|M\times V}^{\otimes-1} \\
&\simeq \roim{q_1}\epb{j'}\rhom(\opb{q_2}F,\epb{s}G)
  \otimes \omega_{N\times V|M\times V}^{\otimes-1} \\
&\simeq \roim{q_1}\rhom(\opb{j'}\opb{q_2}F,\epb{j'}\epb{s}G)
  \otimes \omega_{N\times V|M\times V}^{\otimes-1} \\
&\simeq \roim{q_1}\rhom(\opb{q_2}\opb{j}F,\epb{s}\epb{j}G)
  \otimes \omega_{N\times V|M\times V}^{\otimes-1} \\
&\simeq \rhomc(\opb{j}F, \opb{j}G).
\end{align*}
\end{proof}

\subsubsection*{Kernels associated with cones}
Recall that we consider a trivial vector bundle $E=M\times V$
and  a trivial cone $\gamma=M\times\gamma_0$ satisfying~\eqref{eq:hypcone}.
For another proper closed convex cone $\lambda_0\subset V$  such that
$\lambda_0 \subset \gamma_0$, setting $\lambda=M\times\lambda_0$,
 we shall use the exact sequence of sheaves:
\eq\label{eq:exscone}
&&0\to\cor_{\gamma\setminus \lambda}\to \cor_\gamma\to\cor_\lambda \to 0.
\eneq
\begin{lemma}\label{le:sscone} 
Let $\lambda_0 \subset \gamma_0$ be closed convex proper cones. Then
\eqn
&&\SSi(\cor_\gamma)\subset T^*_MM\times V\times\gamma_0^\circ,\\
&&\SSi(\cor_{\gamma\setminus \lambda})\subset 
T^*_MM\times V\times( \lambda_0^\circ  \setminus\Int(\gamma_0^\circ)).
\eneqn
\end{lemma}
\begin{proof}
Since our sheaves are inverse images of sheaves on $V$ we may as well
assume that $M$ is a point.  Since our sheaves are conic in the sense
of~\cite[\S 5.5]{KS90} their microsupports are biconic. Now, a closed
biconic subset $A$ of $V\times V^*$ satisfies
$A\subset V\times (A\cap \{0\}\times V^*)$.
Hence we only have to check the inclusions at the origin.

Then the first inclusion follows from~\cite[Prop.~5.3.1]{KS90}.

For the second one we use the Sato-Fourier transform
$(\cdot)^\wedge\cl \Derb_{\R^+}(\cor_V) \to \Derb_{\R^+}(\cor_{V^*})$
defined in~\cite[\S 3.7]{KS90} ($\Derb_{\R^+}(\cor_V)$ denotes the
subcategory of complexes with conic cohomology).
We have $(\cor_{\gamma_0})^\wedge \simeq \cor_{\Int \gamma_0^\circ}$
and we deduce the distinguished triangle
$$
(\cor_{\gamma_0\setminus \lambda_0})^\wedge \to 
\cor_{\Int \gamma_0^\circ} \to \cor_{\Int \lambda_0^\circ} \to[+1].
$$
Hence $(\cor_{\gamma_0\setminus \lambda_0})^\wedge \simeq
\cor_{\Int \lambda_0^\circ \setminus \Int \gamma_0^\circ} [-1]$
and we conclude with~\cite[Prop.~5.5.5]{KS90} which implies
$\SSi(F) \cap T^*_0V = \supp(F^\wedge)$ for $F\in  \Derb_{\R^+}(\cor_V)$.
\end{proof}
We introduce the kernel:
\eq\label{eq:Pgamma}
&&\LG\eqdot \cor_\gamma\star\cl \Derb(\cor_E)\to\Derb(\cor_E).
\eneq
The morphism $\cor_\gamma\to\cor_{\{0\}}$ 
induces a morphism of functors $\varepsilon \cl \LG\to\id_{ \Derb(\cor_E)}$. 
By~\eqref{eq:convconv1} we have $\LG\circ \LG\simeq \LG$. 
Hence, the pair $(\LG,\varepsilon)$ is a projector in
$\Derb(\cor_E)^\op$ in the sense of \cite[Chap.~5]{KS05}.
It will be convenient to write $\LG$ with the language of kernels as
in~\eqref{eq:conv00}.
We define $\gamma^+ \subset E\times E$ by
\eq\label{eq:gamma+}
&&\gamma^+ = \{ (x,v,x',v') \in E\times E; \; v-v' \in \gamma_0\}.
\eneq
Then 
\eq\label{eq:Pgamma-bis}
&&\LG \simeq \cor_{\gamma^+} \conv \cdot.
\eneq
In the sequel we set 
\eq\label{eq:UCgamma}
&&\ba{rcl}
&&U_\gamma\eqdot T^*M\times V\times \Int(\gamma_0^\circ),\label{eq:Ugamma}\\
&& Z_\gamma\eqdot T^*E\setminus U_\gamma.
\ea\eneq
\begin{proposition}\label{pro:conv2}
Let $F\in\Derb(\cor_E)$.
\bnum
\item
$\SSi(\LG F)\subset \ol{U_\gamma}=T^*M\times V\times\gamma_0^\circ$.
\item
Consider a distinguished triangle $\LG F\to F\to G\to[+1]$. Then 
$\SSi(G)\subset Z_\gamma$. In particular, 
$\SSi(\LG F)\subset  (T^*M\times V\times\partial\gamma_0^\circ)
\cup (\SSi(F)\cap U_\gamma)$.
\item
Let $G\in\Derb_{Z_\gamma}(\cor_E)$. 
Then $\roim{q}\rsect_\gamma(G)\simeq0$. In particular, $\rsect_\gamma(E;G)\simeq0$.
\enum
\end{proposition}
\begin{proof}
(i) follows from~\eqref{eq:SSstar1b} and Lemma~\ref{le:sscone}. 

\medskip
\noindent
(ii) Using the exact sequence~\eqref{eq:exscone}, we have 
$G\simeq \cor_{\gamma\setminus\{0\}}\star F$.
Then the result again follows from~\eqref{eq:SSstar1b}
and Lemma~\ref{le:sscone}.

\medskip
\noindent
(iii) We set $H =\rsect_\gamma(G)\simeq\rhom(\cor_\gamma,G)$. It follows
from Theorem~\ref{th:opboim2} that $\SSi(H)\subset Z_\gamma$.   Choose a
vector $\xi\in\Int(\gamma_0^\circ)$ and consider the projection
\eqn
&&\theta\cl M\times V \to M\times \R,\quad\theta(x,v)=(x;\langle\xi,v\rangle).
\eneqn
Since $\gamma$ is a proper cone,
$\theta$ is proper on $\supp H$ and  we get by Theorem~\ref{th:opboim} that
$\SSi(\roim{\theta}(H))\subset\{\tau\leq 0\}$ where $(t;\tau)$ are the
coordinates on $T^*\R$. Moreover, 
$\supp\roim{\theta}(H)\subset M\times\{t\geq0\}$.

Now it is enough to prove that $\rsect(U\times\R; \roim{\theta}(H)) = 0$, for
any open subset $U$ of $M$.  Denote by $p\cl U\times\R\to\R$ the projection
and set $\tw H=\roim{p}\roim{\theta}(H)$.  Although $p$ is not proper on
$\supp(\tw H)$, one easily checks that
$\SSi(\tw H)\subset\{t\geq0,\tau\leq0\}$ and this implies
$\tw H\simeq 0$. (This is a special case of Corollary~\ref{cor:Morse}.)
\end{proof}
The next Lemma follows immediately from the adjunction
formula~\eqref{eq:adjstarhomc}.
\begin{lemma}\label{le:secthomc}
Let $F,G\in\Derb(\cor_E)$ and assume that $\LG F\isoto F$. Then we have
$\Hom[\Derb(\cor_E)](F,G)\simeq \rsect_\gamma(E;\rhomc(F,G))$.
\end{lemma}

\begin{proposition}\label{pro:conv3}
\banum
\item
Let $F\in\Derb(\cor_E)$. Then $F\in \Derb_{Z_\gamma}(\cor_E)^{\perp,l}$ if and
only if the natural morphism $\LG F\to F$ is an isomorphism.
\item
Let $G\in\Derb_{Z_\gamma}(\cor_E)$. Then $\LG G\simeq0$.
\eanum
\end{proposition}
\begin{proof}
(a)-(i) Assume $F\simeq \LG F$. Let $G\in\Derb_{Z_\gamma}(\cor_E)$ and
set $H\eqdot\rhomc(F,G)$. Then $H$ belongs to $\Derb_{Z_\gamma}(\cor_E)$
by~\eqref{eq:SSstar1b} and $\rsect_\gamma(E;H)\simeq0$ by
Proposition~\ref{pro:conv2}. Since $F\simeq \LG F$, we get
$\Hom[\Derb(\cor_E)](F,G) =0$ by Lemma~\ref{le:secthomc}.

\smallskip
\noindent
(a)-(ii) Assume that $F\in \Derb_{Z_\gamma}(\cor_E)^{\perp,l}$ and consider a
distinguished triangle $\LG F\to F\to G\to[+1]$. By~(a)-(i) $\LG F$ also
belongs to $\Derb_{Z_\gamma}(\cor_E)^{\perp,l}$. Hence so does $G$. On the
other hand, $G\in\Derb_{Z_\gamma}(\cor_E)$ by
Proposition~\ref{pro:conv2}. Hence, $G\simeq0$.

\medskip
\noindent
(b) Let $G\in\Derb_{Z_\gamma}(\cor_E)$ and consider a distinguished triangle
$\LG G\to G\to H\to[+1]$. Since both $G$ and $H$ belong to
$\Derb_{Z_\gamma}(\cor_E)$, so does $\LG G$. Since $\LG G$ belongs to 
$\Derb_{Z_\gamma}(\cor_E)^{\perp,l}$, it is $0$.
\end{proof}
\begin{remark}\label{rem:LGRG}
One can also consider the  projector 
\eq
&&\RG\eqdot \rhomc(\cor_\gamma,\scbul)
\cl\Derb(\cor_E)\to\Derb(\cor_E).\label{eq:Rgamma}
\eneq
Then we obtain similar results to Propositions~\ref{pro:conv2},~\ref{pro:conv3}
and Lemma~\ref{le:secthomc} with $\RG$ instead of $\LG$.
Note that the pair $(\LG,\RG)$ is a pair of adjoint functors:
\eqn
\Hom[\Derb(\cor_E)](\LG F,G)&\simeq&\Hom[\Derb(\cor_E)](F,\RG G)\\
&\simeq&\Hom[\Derb(\cor_E)](\cor_\gamma,\rhomc(F,G)).
\eneqn
Note that $\cor_\gamma$ is cohomologically constructible. If 
we assume that $\Int(\gamma)\not=\emptyset$, then
$\RD'\cor_\gamma\simeq\cor_{\Int(\gamma)}$ and  one 
deduces from Lemma~\ref{lem:convhom2} that
\eq\label{eqhomcgamma}
&&\rhomc(\cor_\gamma,\cor_\gamma)\simeq\cor_{\Int(-\gamma)}[d_V],
\eneq
where $d_V$ is the dimension of $V$.
\end{remark}

\subsubsection*{Projector and localization}
Recall that $E=M\times V$ is a trivial vector bundle over $M$, $\gamma_0$ is a
cone satisfying~\eqref{eq:hypcone} and the sets ${U_\gamma}$ and ${Z_\gamma}$
are defined in~\eqref{eq:UCgamma}.  By definition $\Derb(\cor_E;{U_\gamma})$
is a localization of $\Derb(\cor_E)$ and we let
$Q_\gamma\cl \Derb(\cor_E)\to\Derb(\cor_E;U_\gamma)$ be the functor of localization.

\begin{proposition}\label{pro:Piota}
\bnum
\item 
The functor $\LG$ defined in~\eqref{eq:Pgamma}
takes its values in $\Derb_{Z_\gamma}(\cor_E)^{\perp,l}$
and sends $\Derb_{Z_\gamma}(\cor_E)$ to $0$.
It factorizes through $Q_\gamma$ and induces a functor
$\lG\cl\Derb(\cor_E;{U_\gamma}) \to \Derb(\cor_E)$ such that 
$\LG\simeq\lG\circ Q_\gamma$.
\item
The functor $\lG$ is left adjoint
to $Q_\gamma$ and induces an equivalence \\
$\Derb(\cor_E;{U_\gamma})\simeq  \Derb_{{Z_\gamma}}(\cor_E)^{\perp,l}$.
\enum
\end{proposition}
This is visualized by the diagram
\eq\label{eq:diagPomega}
&&\vcenter{\xymatrix{
\Derb_{Z_\gamma}(\cor_E)\ar@{^(->}[r]
                  &\Derb(\cor_E)\ar[r]^-{Q_\gamma}\ar[rd]|-{\LG}
&\Derb(\cor_E;U_\gamma)\ar[d]^-{\lG}_-\sim\\
&&\Derb_{Z_\gamma}(\cor_E)^{\perp,l}.
}}
\eneq
\begin{proof}
  This follows from Proposition~\ref{pro:conv3} together with the classical
  results on the localization of triangulated categories recalled in
  Section~\ref{section:mts} (see {\em e.g.,}~\cite[Exe.~10.15]{KS05}).
\end{proof}
In particular, we have for  $F,G \in \Derb(\cor_E)$
\eq\label{eq:Piota}
&&\ba{rcl}
\Hom[\Derb(\cor_E;{U_\gamma})](Q_\gamma(F), Q_\gamma(G))
&\simeq& \Hom[\Derb(\cor_E)](\LG(F),G)\\
&\simeq& \Hom[\Derb(\cor_E)](\LG(F),\LG(G)).
\ea
\eneq
There is a similar result to Proposition~\ref{pro:Piota}, replacing the functor $\LG$ 
with the functor $\RG$. The functor $\RG$
takes its values in $\Derb_{Z_\gamma}(\cor_E)^{\perp,r}$
and sends $\Derb_{Z_\gamma}(\cor_E)$ to $0$.
It factorizes through $Q_\gamma$ and induces a functor
$\rG\cl\Derb(\cor_E;{U_\gamma}) \to \Derb(\cor_E)$ such that 
$\RG\simeq\rG\circ Q_\gamma$. 

We notice that, for  $F\in\Derb_{Z_\gamma}(\cor_E)^{\perp,l}$ or  
$G\in\Derb_{Z_\gamma}(\cor_E)^{\perp,r}$, we have
\eq\label{eq:rhomcperpr}
\rhomc(F,G)\in\Derb_{Z_\gamma}(\cor_E)^{\perp,r} .
\eneq
By Proposition~\ref{pro:conv2} (used with $\RG$ instead of $\LG$) we obtain
in particular
\eq\label{eq:SSrhomc}
\rhomc(F,G) \in \Derb_{\ol{U_\gamma}}(\cor_M).
\eneq

\begin{notation}\label{not:Derbgamma}
Let us set for short
\eq\label{eq:Derbgamma}
&&\ba{rcl}
\Derb(\corg_M)&\eqdot&\Derb(\cor_E;{U_\gamma}),\\
\Derb(\corgl_M)&\eqdot&\Derb_{Z_\gamma}(\cor_E)^{\perp,l},\\
\Derb(\corgr_M)&\eqdot&\Derb_{Z_\gamma}(\cor_E)^{\perp,r}.
\ea\eneq
When $M=\rmpt$, we set
\eq
\Derb(\corg)&\eqdot&\Derb(\corg_\rmpt)
\eneq
and similarly with $\Derb(\corgl)$ and $\Derb(\corgr)$.
\end{notation}

Denote by $p\cl E=M\times V\to V$ the projection
and denote by $\Gamma^\gamma$ the functor
\eq
&&\Gamma^\gamma(\scbul)=\RHom(\cor_\gamma,\scbul)\cl \Derb(\corg)\to\Derb(\cor).
\eneq
We get the diagram of categories in which the horizontal arrows are equivalences
\eq&&\ba{rcl}
\xymatrix{
\Derb(\cor_M^{\gamma,l})\ar[d]^-{\reim p}
   &\Derb(\cor_M^\gamma)\ar[l]_-{\lG}^-\sim\ar[r]^-{\rG}_-\sim
                     &\Derb(\cor_M^{\gamma,r})\ar[d]^-{\roim p}\\
\Derb(\cor^{\gamma,l})
  &\Derb(\cor^\gamma)\ar[l]_-{\lG}^-\sim\ar[r]^-{\rG}_-\sim\ar[d]^-{\Gamma^\gamma}
                 &\Derb(\cor^{\gamma,r})\\
    &\Derb(\cor).&
 }\ea\eneq
Note that
by Lemma~\ref{le:secthomc},
for $F\in \Derb(\corgl_M)$ or $G\in \Derb(\corgr_M)$, we have
\eq\label{eq:hominDgamma}
&&\RHom[\Derb(\corg_M)](F,G)\simeq\Gamma^\gamma\circ\roim{p}\rhomc(F,G).
\eneq

\subsubsection*{Embedding the category $\Derb(\cor_M)$ into $\Derb(\corg_M)$}
Recall that $q\cl E\to M$ denotes the projection and consider the functor 
\eqn
&&\Psi_\gamma\cl \Derb(\cor_M)\to\Derb(\cor_E),\quad
F\mapsto \opb{q}F\tenso\cor_\gamma.
\eneqn
\begin{lemma}\label{le:dbMtodbE}
One has the isomorphism of functors $\LG\circ\Psi_\gamma\isoto\Psi_\gamma$.
\end{lemma}
\begin{proof}
One has 
\eqn
\LG\circ\Psi_\gamma(F)
&=&\reim{s}(\opb{q_1}\cor_\gamma\tenso\opb{q_2}(\opb{q}F\tenso\cor_\gamma))\\
&\simeq&
\reim{s}(\opb{q_1}\cor_\gamma\tenso\opb{q_2}\cor_\gamma\tenso\opb{q_2}(\opb{q}F))\\
&\simeq&
\reim{s}(\opb{q_1}\cor_\gamma\tenso\opb{q_2}\cor_\gamma\tenso\opb{s}(\opb{q}F))\\
&\simeq&
\reim{s}(\opb{q_1}\cor_\gamma\tenso\opb{q_2}\cor_\gamma)\tenso\opb{q}F\\
&\simeq&
\cor_\gamma\tenso\opb{q}F.
\eneqn
\end{proof}
In the sequel, we consider $\Psi_\gamma$ as a functor 
\eq\label{eq:dbMtodbEb}
&&\Psi_\gamma\cl \Derb(\cor_M)\to\Derb(\corg_M).
\eneq
\begin{proposition}\label{pro:dbMtodbE}
The functor $\Psi_\gamma$ in~\eqref{eq:dbMtodbEb} is fully faithful.
\end{proposition}
\begin{proof}
Let $F,G\in\Derb(\cor_M)$. Then
\eqn
\Hom[\Derb(\cor_E)](\cor_\gamma\tenso\opb{q}G,\cor_\gamma\tenso\opb{q}F)&&\\
&&\hspace{-9ex}
\simeq\Hom[\Derb(\cor_M)](G,\roim{q}\rhom(\cor_\gamma,\opb{q}F\tenso\cor_\gamma))\\
&&\hspace{-9ex}
\simeq\Hom[\Derb(\cor_M)](G,\roim{q}(\opb{q}F\tenso\cor_\gamma)).
\eneqn
Hence, it is enough to check the isomorphism
\eq\label{eq:isoopbqoimq}
&&F\isoto\roim{q}(\opb{q}F\tenso\cor_\gamma).
\eneq
Denote by $\tw q$ the projection $\gamma\to M$.
The isomorphism~\eqref{eq:isoopbqoimq} reduces to
\eqn
&&F\simeq\roim{\tw q}\opb{\tw q}F
\eneqn
and this last isomorphism follows from the fact that $\gamma$ is a
closed convex cone, hence is contractible (see for example~\cite[Prop.~2.7.8]{KS90}).
\end{proof}

\subsubsection*{A cut-off result}
Recall that we consider a trivial vector bundle $E=M\times V$
and  a trivial cone $\gamma=M\times\gamma_0$ satisfying~\eqref{eq:hypcone}.
We also recall that a subset of $T^*M\times V^*$ is a cone if it is stable by
the action~\eqref{eq:defcone}.
 The map $\piw$ is defined in~\eqref{eq:projfibre} and we have set
 (see~\eqref{eq:UCgamma}):
\eqn
&&U_\gamma= T^*M\times V\times \Int\gamma_0^\circ.
\eneqn
By the equivalence $\lG$ of Proposition~\ref{pro:Piota}, any object
$F\in\Derb(\cor_E;{U_\gamma})$ has a canonical representative in
$\Derb(\cor_E)$ again denoted by $F$ and we have $F\simeq \LG(F)$. By
Proposition~\ref{pro:conv2}~(i) we have $\SSi(F) \subset \ol{U_\gamma}$.

We first state a kind of cut-off lemma in the case where $M$ is a point. 
\begin{lemma}\label{le:cutoff0}
Let $V$ be a vector space and $\gamma\subset V$ a closed convex proper
cone containing $0$.
Set $U_\gamma\eqdot V \times \Int\gamma^\circ$ and
$Z_\gamma\eqdot T^*V\setminus U_\gamma$.
Let $F\in\Derb_{Z_\gamma}(\cor_V)^{\perp,l}$.
We assume that there exists a closed cone $A \subset V^*$ such that
\bnum
\item  $A \subset \Int\gamma^\circ\cup \{0\}$,
\item $\SSi(F)\cap U_\gamma \subset  V\times A$.
\enum
Then $\SSi(F)\subset (\SSi(F)\cap U_\gamma) \cup T^*_VV$.
\end{lemma}
\begin{proof}
(i) Up to enlarging $A$ we may as well assume that 
$\SSi(F)\cap U_\gamma \subset V\times \Int A$. We set $\lambda = A^\circ$.
Hence $\lambda$ is a closed convex proper cone of $V$ and we have
\eq
\label{eq:cutoff1}
&\lambda^\circ \setminus \{0\} \subset \Int(\gamma^\circ), \\
&\label{eq:cutoff2}
\SSi(F)\cap U_\gamma \subset V\times \Int(\lambda^\circ).
\eneq
We will prove that $L_\lambda(F)$ satisfies the conclusion of the lemma
as well as the isomorphism $L_\lambda(F) \isoto F$.

\medskip\noindent (ii)
By~\eqref{eq:cutoff2} and Proposition~\ref{pro:conv2}~(ii) we have
\eq\label{eq:cutoff3}
\SSi(F)\subset V\times (\partial\gamma^\circ \cup \Int(\lambda^\circ)).
\eneq
By~\eqref{eq:SSstar1b} we deduce
\eqn
 \SSi(L_\lambda F)
&\subset & V\times (\lambda^\circ \cap (\partial\gamma^\circ\cup \Int(\lambda^\circ)) )\\
&=& V\times (  \Int(\lambda^\circ) \cup \{0\} )\\
&\subset & U_\gamma \cup T^*_VV .
\eneqn

\medskip\noindent (iii)
It remains to see that $F \simeq L_\lambda(F)$. We consider the
distinguished triangle
$\cor_{\lambda \setminus \gamma}\star F \to L_\lambda F \to \LG F \to[+1]$.
We have $\LG F \isoto F$.
By~\eqref{eq:cutoff3}, Lemma~\ref{le:sscone} and~\eqref{eq:SSstar1b}
we have
\eqn
\SSi(\cor_{\lambda\setminus\gamma} \star F)
\subset V\times
( (\gamma^\circ  \setminus \Int(\lambda^\circ) ) \cap
(\partial \gamma^\circ \cup \Int(\lambda^\circ) ))
\subset Z_\gamma,
\eneqn
which shows that $L_\lambda F\to F$ is an isomorphism in
$\Derb(\cor_V;U_\gamma)$.  By Proposition~\ref{pro:Piota} we
obtain $F\simeq \LG (L_\lambda F)$.  But
$\cor_\gamma \star \cor_\lambda \simeq \cor_\lambda$ and we get
finally $F \simeq L_\lambda F$.
\end{proof}
Now we extend Lemma~\ref{le:cutoff0} to the case of an arbitrary manifold $M$.
We consider a finite dimensional real vector space $E=E'\times E''$
with $E'=\R^d$.
We write $x=(x',x'')\in E'\times E''$ and $x'=(x'_1,\dots,x'_d)\in\R^d$.
We set $U = ]-1,1[^d\times E''$.
We choose a diffeomorphism $\phi\cl ]-1,1[ \isoto \R$ such that
$d\phi(t) \geq 1$ for all $t\in ]-1,1[$ and we define
\eqn
&&\Phi\cl U \isoto E,\quad \Phi(x'_1,\ldots,x'_d,x'')=(\phi(x'_1),\ldots,\phi(x'_d),x'').
\eneqn
\begin{lemma}\label{lem:cone-chgt-coord}
In the preceding situation, consider two closed convex proper cones
$\gamma_0\subset E''$ and $C_1\subset E^*$ such that
$C_1 \subset (E'^*\times \Int(\gamma_0^\circ))\cup \{(0,0)\}$.
Then there exists another closed convex proper cone $C_2\subset E^*$ such that
$C_2 \subset (E'^*\times \Int(\gamma_0^\circ))\cup \{(0,0)\}$ and
\eqn
&&\Phi_\pi \opb{\Phi_d} (U\times C_1 )\subset E\times C_2.
\eneqn
\end{lemma}
\begin{proof}
(i) We assume that $\Int(\gamma_0^\circ)$ is non empty (otherwise the lemma 
is trivial).  Then a closed cone of $E^*$ is contained in
$(E'^*\times \Int(\gamma_0^\circ))\cup \{(0,0)\}$ if and only if it is
contained in $C_{a,D} \eqdot \R_{\geq 0}\cdot ([-a,a]^d\times D)$ for some
$a>0$ and some compact subset $D \subset \Int(\gamma_0^\circ)$.
Hence we may assume $C_1 = C_{a,D}$.

\medskip\noindent
(ii) Denote by $(x';\xi')$ the coordinates on $\R^d\times(\R^d)^*$.  We may
assume that $E''=\R^{m}$ and we denote by $(x'';\xi'')$ the coordinates on
$E''\times(E'')^*$. The change of coordinates $\Phi$ defined by
$y'_i = \phi(x'_i)$ ($i=1,\dots,d$), $y''= x''$ associates the coordinates
$(y;\eta)=(y',y'';\eta',\eta'')$ to the coordinates
$(x'_1,\dots,x'_d,x''$; $\xi'_1,\dots,\xi'_d,\xi'')$ with
\eqn
&&y'_i=\phi(x'_i),\quad \eta'_i= d\phi^{-1}(x'_i)\cdot\xi'_i,
\quad  (i=1,\dots,d),\\
&& y''=x'',\quad \eta''=\xi''.
\eneqn
Since $d\phi(t) \geq 1$, we get that
$\Phi_\pi \opb{\Phi_d}(U\times C_{a,D})\subset E\times C_{a,D}$
and we may choose $C_2=C_{a,D}$.
\end{proof}
\begin{theorem}\label{th:cutoff}
Let $F\in\Derb_{Z_\gamma}(\cor_E)^{\perp,l}$. We assume that there exists
$A \subset T^*M\times V^*$ such that
\bnum
\item $A$ is a closed strict $\gamma$-cone
\lp see Definition~\ref{def:strictcone}\rp,
\item $\SSi(F)\cap U_\gamma \subset \opb{\piw_E}(A)$.
\enum
Then $\SSi(F)\subset (\SSi(F)\cap U_\gamma) \cup T^*_EE$.
\end{theorem}
\begin{proof}
Since the statement is local on $M$ we may assume that $M$ is an open 
subset of a vector space $W$.
Then $A$ is a closed  subset of $M\times W^* \times V^*$. For any
$x\in M$, $A_x\eqdot A \cap (\{x\} \times W^* \times V^*)$ is a cone
satisfying
\eqn
&&A_x \subset (W^* \times \Int(\gamma_0^\circ)) \cup \{(0,0)\}.
\eneqn
For $x_0\in M$ and for a given compact neighborhood $C$ of $x_0$ we may assume that there
exists a closed convex cone $B$ of $W^* \times V^*$ such that
$A_x \subset B$ for any $x\in C$ and
\eqn
&&B \subset (W^* \times \Int(\gamma_0^\circ)) \cup \{(0,0)\}.
\eneqn
We may assume $x_0=0\in W$. 
We choose an isomorphism $W\simeq \R^d$ so that $]-1,1[^d\subset C$.
Then we apply a change of coordinates as in Lemma~\ref{lem:cone-chgt-coord},
with $E'=W$, $E''=V$, $C_1=B$, and we are reduced to Lemma~\ref{le:cutoff0}
applied to the vector space $W\times V$ and the cone
$\gamma = \{0\}\times \gamma_0$.
\end{proof}

\subsubsection*{A separation theorem}

The next result is a slight generalization of Tamarkin's Theorem~\cite[Th.~3.2]{Ta}. 
In this statement and its proof, we write $\piw$ instead of $\piw_E$ for short.
\begin{theorem}{\rm (The separation theorem.)}\label{th:separation}  
Let $A,B$ be two closed strict $\gamma$-cones in $T^*M\times V^*$.
Let $F\in\Derb_{ \opb{\piw}(A)}(\cor_E;{U_\gamma})$ and  
$G\in\Derb_{ \opb{\piw}(B)}(\cor_E;{U_\gamma})$. 
Assume that $A\cap B \subset T^*_MM\times \{0\}$ and that the projection
$q_2 \cl M\times V \to V$ is proper on the set
$\{(x,v_1-v_2);\, (x,v_1) \in \supp G, (x,v_2)\in \supp F \}$.
Then 
$$
\roim{q_2} \rhomc(\lG(F), \lG(G)) \simeq 0 ,
$$
where $\lG$ is defined in Proposition~\ref{pro:Piota}.
In particular $\Hom[\Derb(\cor_E;{U_\gamma})](F,G)\simeq0$.
\end{theorem}
\begin{proof}
We set $L = \rhomc(\lG(F), \lG(G))$ and $L' = \roim{q_2}L$.
By~\eqref{eq:rhomcperpr} we have $L \in\Derb_{Z_\gamma}(\cor_E)^{\perp,r}$.
By adjunction between $\roim{q_2}$ and $\opb{q_2}$ we deduce
$L' \in\Derb_{Z_{\gamma_0}}(\cor_V)^{\perp,r}$. 
It remains to check that
$\SSi(L') \subset Z_{\gamma_0}$. 

By Theorem~\ref{th:cutoff} we have
$\SSi(F)\subset \opb{\piw}(A)$ and $\SSi(G)\subset \opb{\piw}(B)$.
Then Proposition~\ref{pro:SSconv} gives
$\SSi(L) \subset \opb{\piw}(A)\hatstar (\opb{\piw}(B))^\alpha$.
Applying Lemma~\ref{lem:star-gamma-cones} we get 
\eqn
&&\SSi(L) \cap (T^*_MM \times T^*V) \subset T^*_EE.
\eneqn
Using Lemma~\ref{lem:convhom2}, the hypothesis implies that $q_2$ is proper on
$\supp L$. We deduce $\SSi(L') \subset T^*_VV$
and thus $L'\simeq 0$.

\medskip
\noindent
Proposition~\ref{pro:Piota} and  Lemma~\ref{le:secthomc} give
the first two isomorphisms in the sequence
\begin{multline*}
\Hom[\Derb(\cor_E;{U_\gamma})](F,G)
\simeq  \Hom[\Derb(\cor_E)](F, G)   \\
\simeq   \Hom[\Derb(\cor_E)](\cor_\gamma,L)
\simeq   \Hom[\Derb(\cor_V)](\cor_{\gamma_0},L') \simeq 0,
\end{multline*}
which proves the last assertion.
\end{proof}

\subsubsection*{Kernels}
We consider $E=M\times V$, $\gamma = M\times \gamma_0$ and a kernel $K\in
\Derb(\cor_{E\times E})$.  We introduce the coordinates
$(x,y,x',y';\xi,\eta,\xi',\eta')$ on $T^*(E\times E)$ and we make the
following hypothesis
\eq\label{eq:hyp-kernel}
\SSi(K) \subset \{ \eta + \eta' =0 \} .
\eneq
We recall that $\LG \simeq \cor_{\gamma^+} \conv \cdot$, where
$\gamma^+ \subset E\times E$ is defined in~\eqref{eq:gamma+}.
\begin{proposition}\label{prop:action-kernel}
Let $K\in \Derb(\cor_{E\times E})$ which satisfies~\eqref{eq:hyp-kernel}.
Then $K \conv \cor_{\gamma^+} \simeq \cor_{\gamma^+} \conv K$.
In particular $K \conv \cdot$ sends $\Derb(\corgl_M)$ into itself.
Moreover $\SSi(K) \aconv \{\eta <0\} \subset \{\eta <0\}$ and
$\SSi(K) \aconv \{\eta \geq 0\} \subset \{\eta \geq 0\}$.
\end{proposition}
\begin{proof}
We define the projection
$\sigma\cl M\times V \times M \times V \to M\times M \times V$
as the product of $\id_{M\times M}$ with $\sigma_0\cl V\times V \to V$,
$(y,y') \mapsto y-y'$.
Then the hypothesis~\eqref{eq:hyp-kernel} and Corollary~\ref{cor:Morse}
give $K\simeq \opb{\sigma}(K')$, where $K' = \roim{\sigma}(K)$.
We also have by definition
$\cor_{\gamma^+} \simeq \opb{\sigma}(\cor_{M\times M \times \gamma_0})$.
The base change formula applied to the Cartesian square
\eqn
&&\xymatrix@C=2cm{
V\times V\times V \ar[d]_{\sigma_0 \circ q_{12} \times \sigma_0 \circ q_{23}}
\ar[r]^{q_{13}}
& V\times V \ar[d]^{\sigma_0}  \\
 V\times V \ar[r]^s &  V}
\eneqn
gives the first and third isomorphisms below:
\eqn
K \conv \cor_{\gamma^+}
\simeq \opb{\sigma}(K'\star \cor_{M\times M \times \gamma_0})
\simeq \opb{\sigma}(\cor_{M\times M \times \gamma_0} \star K')
\simeq \cor_{\gamma^+} \conv K.
\eneqn
The last assertion follows from the hypothesis~\eqref{eq:hyp-kernel}.
\end{proof}

\section{The Tamarkin category}
\label{section:tam1}
We particularize the preceding results to the case where $V=\R$ and
$\gamma_0=\{t\in\R;t\geq0\}$. 
Hence, with the notations of~\eqref{eq:UCgamma}, we have $U_\gamma=\{\tau>0\}$.
As in Section~\ref{section:GKS} 
we denote by  $T^*_{\{\tau>0\}}(M\times\R)$ the open subset $\{\tau>0\}$ of
$T^*(M\times\R)$ and we define the map 
\eq\label{eq:rho2}
&&\rho\cl T^*_{\{\tau>0\}}(M\times\R)\to T^*M,\quad (x,t;\xi,\tau)\mapsto
(x;\xi/\tau).
\eneq
We also use Notations~\ref{not:Derbgamma}. Moreover, for a closed subset $A$
of $T^*M$ we set
$$
\Derb_{A}(\cort_M)\eqdot\Derb_{\opb{\rho}(A)}(\cor_{M\times\R};\{\tau>0\}).
$$
\begin{lemma}\label{lem:Acompact}
Let $A \subset T^*M$ and $F\in\Derb_A(\cort_M)$.
Let $A' \subset T^*M \times \R$ be given by
$A' = \{(x;\xi,\tau); \tau>0, \; (x;\xi/\tau) \in A\}$
and consider $F$ as an object of $\Derb(\corgl_M)$.
Assume that $\pi_M$ is proper on $A$.
Then $\ol{A'}$ is a strict $\gamma$-cone and
$\SSi(F)\subset \opb{\piw}(\ol{A'})$. In particular
$\supp(F) \subset \pi_M(A) \times \R$.
\end{lemma}
\begin{proof}
The properness hypothesis gives
$\ol{A'} = A' \cup (\pi_M(A) \times \{\tau=0\})$ and this implies the
first assertion.
Then Theorem~\ref{th:cutoff} gives
$\SSi(F)\subset \opb{\piw}(\ol{A'}) \cup T^*_{M\times\R}(M\times\R)$. 
Hence, if $(x,t;0,0)\not\in \opb{\piw}(\ol{A'})$, we have
$\SSi(F|_{U\times \R}) \subset T^*_{U\times \R}(U\times\R)$ for
some neighborhood $U$ of $x$. But $\LG F \simeq F$
and we deduce $F|_{U\times \R} =0$, which proves
$(x,t;0,0)\not\in \SSi(F)$. So we get $\SSi(F)\subset \opb{\piw}(\ol{A'})$. 
\end{proof}

\begin{example}\label{exa:shtA}
(i) Let $M=\R$ endowed with the coordinate $x$ and consider the set
\eqn
&&Z=\{(x,t)\in M\times\R;-1\leq x\leq 1, 0\leq 2t<-x^2+1\}.
\eneqn
Consider the sheaf $\cor_Z$ and denote by $(x,t;\xi,\tau)$
the coordinates on  $T^*(M\times\R)$.
The set $\SSi(\cor_Z)$ is given by 
\eqn
&
\{t=0,-1\leq x\leq 1, \tau>0,\xi=0\}\cup\{2t=-x^2+1,\xi=x\tau,\tau>0\}\\
&\cup\{x= -1,t=0,0\leq-\xi\leq\tau,\tau>0\}\cup \{x= 1,t=0,0\leq\xi\leq\tau,\tau>0\}\\
&\cup \overline{Z} \times\{\xi=\tau=0\}.
\eneqn
It follows that, denoting by $(x;u=\xi/\tau)$ the coordinates in $T^*M$,
 $\rho(\SSi(\cor_Z) \cap(T^*M\times \dT^*\R))$ is the set 
\eqn
&&\{u=0,-1\leq x\leq 1\}\cup\{u=x,-1\leq x\leq 1\}\\
&&\cup \{x=-1, -1\leq u\leq 0\}\cup \{x=1, 0\leq u\leq 1 \}.
\eneqn

\noindent
(ii) Let $a\in\R$ and consider the set $Z=\{(x,t)\in M\times\R;t\geq ax\}$. 
Then $\rho(\SSi(\cor_Z))$ in $T^*M$ is the set $\{(x;u);u=a\}$.

\noindent
(iii) If $G$ is a sheaf on $M$ and $F=G\etens\cor_{s\geq0}$, then 
$\rho(\SSi(F))=\SSi(G)$. 
\end{example}

\subsubsection*{The separation theorem}
Using Lemma~\ref{lem:Acompact} we get the following particular case of
Theorem~\ref{th:separation}:

\begin{theorem}\label{th:Tamvan}{\rm (see \cite[Th.~3.2]{Ta}.)}
Let $A$ and $B$ be two compact subsets of $T^*M$ and assume that
$A\cap B=\emptyset$. 
Then, for any $F\in\Derb_A(\cort_M)$ and $G\in\Derb_B(\cort_M)$, we have
$\Hom[\Derb(\cort_M)](F,G)\simeq 0$.
\end{theorem}

\section{Localization by torsion objects}\label{section:tam2}
In~\cite{Ta}, Tamarkin introduces the notion of torsion objects, but
does not study the category of such objects systematically. Hence,
most of the results of this section are new.

In this section we set for short $Z = (T^*M)\times\R \times \{\tau\geq 0 \}$,
a closed subset of $T^*(M\times\R)$.
Recall that $\Derb_Z(\cor_{M\times \R})$ is the subcategory of
$F\in \Derb(\cor_{M\times \R})$ such that $\SSi(F) \subset Z$.  By
Proposition~\ref{prop:cut-offlemma} we have $F\in \Derb_Z(\cor_{M\times \R})$
if and only if the morphism~\eqref{eq:convgamma-id} is an isomorphism, which
reads
\eq\label{eq:iso-convgamma-id}
F\npstar \cor_{M\times [0,+\infty[} \isoto F .
\eneq
Define the map 
\eqn
&&T_c\cl M\times\R\to M\times\R, \quad (x,t)\mapsto (x,t+c).  
\eneqn
For $F\in\Derb_Z(\cor_{M\times\R})$ we deduce easily
from~\eqref{eq:iso-convgamma-id}
\eq\label{eq:iso-convgamma-trans}
F\npstar \cor_{M\times [c,+\infty[} \isoto \oim{T_c} F .
\eneq
The inclusions $[d,+\infty[ \subset [c,+\infty[$, for $c \leq d$,
induce natural morphisms  of functors from $\Derb_Z(\cor_{M\times\R})$
to itself
\eqn
&&\tau_{c,d} \cl \oim{T_c} \to \oim{T_d},\quad c \leq d.
\eneqn
We have the identities:
\eq\label{eq:isom_T_tau}
&&\oim{T_c} \circ\oim{T_d} \simeq \oim{T_{(c+d)}},\quad c,d\in\R,\\
&&\oim{T_e}(\tau_{c,d}(\scbul))=\tau_{e+c,e+d}(\scbul)=\tau_{c,d}(\oim{T_e}(\scbul)),
\quad c \leq d, e\in\R,\\
&&\tau_{c,d} \circ\tau_{d,e} = \tau_{c,e}, \quad c \leq d\leq e.
\eneq
\begin{definition}\label{def:torobjb}{\rm (Tamarkin.)}
An object $F\in \Derb_Z(\cor_{M\times\R})$ is called a torsion object
if $\tau_{0,c}(F) =0$ for some $c\geq 0$ (and hence all $c'\geq c$).
\end{definition}
Let $F\in \Derb_Z(\cor_{M\times\R})$ and assume that $F$ is
supported by $M\times [a,b]$ for some compact interval $[a,b]$ of $\R$.
Then $F$ is a torsion object. 
\begin{remark}
One can give an alternative definition of the torsion objects
by using the classical notion of ind-objects (see~\cite{KS05} for an exposition).
An object $F\in\Derb_Z(\cor_{M\times\R})$ is torsion if and only if the natural
morphism $F\to\sinddlim[c] \oim{T_c}F$ is the zero morphism.
\end{remark}
We let $\shnt$ be the full subcategory of $\Derb_Z(\cor_{M\times\R})$ 
consisting of torsion objects.

\begin{lemma}\label{lem:critere-torsion}
Let $F \to[u] G \to[v] H \to[w] F[1]$ be a distinguished triangle in
$\Derb_Z(\cor_{M\times\R})$.
\bnum
\item
If $H$ belongs to $\shnt$, then there exist $c\geq 0$ and 
$\alpha\cl G\to\oim{T_c} F$ 
such that $\tau_{0,c}(F)=\alpha\circ u$.
\item
If there exist $c\geq 0$ and $\alpha\cl G \to\oim{T_c} F$
making the diagram
\eqn
&&\xymatrix@C=2cm{
F \ar[r]^u \ar[d]_{\tau_{0,c}(F)}& G \ar[dl]_\alpha \ar[d]^{\tau_{0,c}(G)}\\
\oim{T_c} F \ar[r]^{\oim{T_c}u} & \oim{T_c} G 
}
\eneqn
commutative, then $H\in\shnt$.
\enum
\end{lemma}
\begin{proof}
(i) Choose $c\geq 0$ such that $\tau_{0,c}(H)\simeq 0$ and consider
the diagram with solid arrows
\eqn
\vcenter{
\xymatrix@C=1.8cm{
H[-1] \ar[r]^{w[-1]} \ar[d]_{\tau_{0,c}(H[-1])} 
             &F\ar[r]^u \ar[d]_{\tau_{0,c}(F)} 
                  &G \ar[d]_{\tau_{0,c}(G)}  \ar@{-->}[dl]_\alpha \\
\oim{T_c} H[-1]\ar[r]^{\oim{T_c}w}  
     &\oim{T_c}F\ar[r]^{\oim{T_c}u}
            &\oim{T_c} G .
}}
\eneqn
Since $\tau_{0,c}(H[-1])\simeq0$,
we have $\tau_{0,c}(F)\circ w[-1] =0$.  Since $\Hom(\scbul,\oim{T_c}F)$ 
is a cohomological functor we deduce the existence of $\alpha$.

\noindent
(ii) We apply $\oim{T_c}$ twice and obtain morphisms of distinguished
triangles:
\eqn
\vcenter{
\xymatrix@C=1.8cm{
F \ar[r]^u \ar[d]_{\tau_{0,c}(F)} 
   & G \ar[r]^v \ar[d]_{\tau_{0,c}(G)}  \ar[dl]_\alpha
        & H \ar[r]^w \ar[d]_{\tau_{0,c}(H)} 
            & F[1] \ar[d]^{\tau_{0,c}(F[1])} \ar@{-->}[dl]_\beta  \\
\oim{T_c}F\ar[r]^{\oim{T_c}u} \ar[d]
   &\oim{T_c} G \ar[r]^{\oim{T_c}v} \ar[d] 
       &\oim{T_c} H \ar[r]^{\oim{T_c}w}  \ar[d]_{\tau_{0,c} (\oim{T_c}H)}
            &\oim{T_c} F[1] \ar[d] \ar@{-->}[dl]_{\oim{T_c}\beta}  \\
\oim{T_{2c}}F \ar[r]  
  & \oim{T_{2c}} G \ar[r]
       & \oim{T_{2c}} H \ar[r]  
            &\oim{T_{2c}} F[1] .
}}
\eneqn
By hypothesis $\tau_{0,c}(H)\circ v=\oim{T_c}v\circ\oim{T_c}u\circ\alpha =0$. 
As above, we deduce the existence of $\beta$ such that
$\tau_{0,c}(H)=\beta\circ w$.  Applying the morphism of functors
$\tau_{0,c}\cl \id\to T_c$ to $\beta$ we find 
\eqn
&&\tau_{0,c}(\oim{T_c}H)\circ\beta=\oim{T_c}\beta\circ\tau_{0,c}(F[1]).  
\eneqn
We deduce:
\begin{multline*}
\tau_{0,c}(\oim{T_c}H) \circ \tau_{0,c}(H) 
= \tau_{0,c}(\oim{T_c}H) \circ \beta \circ w  
= \oim{T_c}\beta \circ \tau_{0,c}(F[1])\circ w \\
= \oim{T_c}\beta \circ \alpha[1] \circ u[1]  \circ w   =0.
\end{multline*}
Using~\eqref{eq:isom_T_tau} we obtain $\tau_{0,2c}(H)\simeq0$ so that $H\in\shnt$.
\end{proof}

\begin{theorem}\label{th:torsion}
The subcategory $\shnt$ is a null system in $\Derb_Z(\cor_{M\times\R})$.
\end{theorem}
\begin{proof}
It is clear that an object isomorphic to a torsion object is itself
a torsion object and that $\shnt$ is stable by the shift
functor. Hence it remains 
to check that if $F\to G\to H \to[+1]$ is a
distinguished triangle with $F,G\in\shnt$ then $H\in\shnt$.
We choose $c\geq 0$ such that $\tau_{0,c}(F) =0$ and $\tau_{0,c}(G)=0$ 
and we apply Lemma~\ref{lem:critere-torsion}~(ii) to the diagram
\eqn
&&\xymatrix@C=2cm{
F \ar[r]^u \ar[d]_{0}  & G \ar[dl]_0 \ar[d]^{0}  \\
\oim{T_c} F \ar[r]^{\oim{T_c}u}  & \oim{T_c} G .
}
\eneqn
\end{proof}
\begin{corollary}\label{cor:cone-tau-torsion}
For any $F\in \Derb_Z(\cor_{M\times\R})$ and any $c\geq 0$, 
the cone of $\tau_{0,c}(F)$ is a torsion object.
\end{corollary}
\begin{proof}
We apply Lemma~\ref{lem:critere-torsion}~(ii) to the commutative diagram
\eqn
&&\xymatrix@C=3cm{
F \ar[r]^{\tau_{0c}(F)}\ar[d]_{\tau_{0c}(F)}  
&\oim{T_c} F \ar[dl]_\id \ar[d]^{\tau_{0c}(\oim{T_c} F )}  \\
\oim{T_c} F\ar[r]^{\oim{T_c}{\tau_{0c}(F)}}  & \oim{T_{2c}} F .
}
\eneqn
\end{proof}

The subcategory $\Derb(\corgl_M)$ of $\Derb(\cor_{M\times\R})$ is
contained in $\Derb_Z(\cor_{M\times\R})$. So we can define torsion
objects in $\Derb(\corgl_M)$ or in the equivalent category $\Derb(\cort_M)$.
We let $\shnt^\gamma$ be the subcategory of torsion objects in
$\Derb(\cort_M)$. Then Theorem~\ref{th:torsion} implies that $\shnt^\gamma$
is a null system.
\begin{definition}\label{def:catTau}
The triangulated category $\sht(\cor_M)$ is the localization of
$\Derb(\cort_M)$ by the null system $\shnt^\gamma$. In other words,
$\sht(\cor_M)=\Derb(\cort_M)/\shnt^\gamma$.
\end{definition}
By Corollary~\ref{cor:cone-tau-torsion}, $\tau_{0,c}(G)$ becomes
invertible in $\sht(\cor_M)$ for any $G\in\Derb(\cort_M)$. 
Hence for a morphism $u\cl F \to G$ in $\Derb(\cort_M)$
and for $c\geq 0$ we can define
$\tau_{0,c}(G)^{-1} \circ u \cl F \to G$ in $\sht(\cor_M)$.  The
family of $\tau_{c,c'}(G)$'s defines an inductive system 
$\{\oim{T_c} G\}_c$ and we
have $\tau_{0,c'}(G)^{-1} \circ \tau_{c,c'}(G) \circ u =
\tau_{0,c}(G)^{-1} \circ u$ for $c'\geq c$.  This defines a natural
morphism:
\eq\label{eq:morphism_sht}
\indlim[c\to+\infty]
\Hom[{\Derb(\cort_M)}](F,\oim{T_c} G)\to\Hom[\sht(\cor_M)](F,G).
\eneq
\begin{proposition}\label{pro:locandlim}
For any $F,G\in \Derb(\cort_{M})$ the morphism~\eqref{eq:morphism_sht}
is an isomorphism.
\end{proposition}
\begin{proof}
(i) Let us first show that~\eqref{eq:morphism_sht} is surjective.  A
morphism $u\cl F \to G$ in $\sht(\cor_M)$ is given by 
$F\to[v] G'\from[s]G$, where the cone of $s$ is a torsion object.  By
Lemma~\ref{lem:critere-torsion}~(i) there exist $c\geq 0$ and
$\alpha\cl G'\to \oim{T_c}G$ such that $\tau_{0,c}(G)=\alpha\circ s$:

\eqn
&&\xymatrix{
F \ar[r]^v  & G' \ar[rd]_\alpha 
& G \ar[l]_s \ar[d]^{\tau_{0,c}(G)} \\
&& \oim{T_c} G .
}
\eneqn
Hence we obtain $u = \tau_{0,c}(G)^{-1} \circ \alpha \circ v$ in
$\sht(\cor_M)$.  In other words $u$ is the image of $\alpha \circ v$
by~\eqref{eq:morphism_sht}.

\medskip
\noindent
(ii) Now we show that~\eqref{eq:morphism_sht} is injective.  We
consider $u\cl F \to \oim{T_c} G$ in
$\Derb(\cort_M)$ such that $\tau_{0c}(G)^{-1}\circ u =0$ 
in $\sht(\cor_M)$. Then $u=0$ in $\sht(\cor_M)$ and this
means that there exists $s\cl \oim{T_c} G \to G'$ such that the cone of
$s$ is a torsion object and $s \circ u =0$ in
$\Derb(\cort_M)$.  By
Lemma~\ref{lem:critere-torsion}~(i) there exist $d\geq 0$ and
$\alpha\cl G'\to \oim{T_{(c+d)}}G$ such that $\tau_{c,c+d}(G)=\alpha\circ s$:
\eqn
&&\xymatrix@C=2cm{
F \ar[r]^u &\oim{T_c}G\ar[d]_{\tau_{c, c+d}(G)} \ar[r]^s& G' \ar[dl]^\alpha \\
& \oim{T_{(c+d)}} G .
}
\eneqn
We obtain $\tau_{c,c+d}(G) \circ u = \alpha \circ s \circ u =0$
which means that the image of $u$ in the left hand side
of~\eqref{eq:morphism_sht} is zero, as required.
\end{proof}
Recall the functor $\Psi_\gamma$ in~\eqref{eq:dbMtodbEb}.
\begin{corollary}\label{cor:dbMtodbE}
The composition
$\Derb(\cor_M) \to[\Psi_\gamma] \Derb(\cor_{M\times\R};U_\gamma) \to \sht(\cor_M)$
is a fully faithful functor.
\end{corollary}
\begin{proof}
For $F,G\in\Derb(\cor_M)$, the proof of Proposition~\ref{pro:dbMtodbE}
gives as well
$$
\Hom[\Derb(\cor_M)](G,F) \isoto
\Hom[\Derb(\cor_{M\times \R})](G\etens \cor_{[0,+\infty[},
F\etens \cor_{[c,+\infty[})
$$
for any $c\geq 0$. Then the result follows from
Proposition~\ref{pro:locandlim}.
\end{proof}

\subsubsection*{Strict cones and torsion}
For a connected manifold $M$ and $F\in \Derb_Z(\cor_{M\times\R})$ we give a
condition on $\SSi(F)$ which implies that $F$ is torsion over any compact
subset as soon as it is torsion at one point.

We first give a preliminary result on $M\times I\times \R$.
We set $E=\R^2$ and we take coordinates
$(s,t;\sigma,\tau)$ on $T^*E$. We fix $\alpha>0$ and define
the cone $\gamma_\alpha = \{ (s,t); t\geq \alpha |s| \}$ in $E$.
We set $U_\alpha = E \times \Int\gamma_\alpha^\circ$.
We recall Proposition~\ref{prop:cut-offlemma}, reformulated
using~\eqref{eq:convgammabis}:
for $F\in \Derb(\cor_{M\times E})$, we have
$\SSi(F)\subset T^*M \times \ol U_\alpha$ if and only if
\eq\label{eq:convgamma-id-ter}
F\npstar \cor_{M\times \gamma_\alpha} 
\simeq \roim{s_E} \rsect_{M\times E \times \Int \gamma_\alpha} (\opb{q_1}F)
\isoto F,
\eneq
where $s_E\cl M\times E\times E \to M\times E$ is the sum of $E$.
\begin{proposition}\label{prop:proj-torsion}
Let $I$ be an interval of $\R$,
$M$ a manifold and $q\cl M\times I\times\R \to M\times \R$ the
projection. Set $\gamma = I\times [0,+\infty[$.
Let $F\in \Derb(\cor_{M\times I\times\R})$.
We assume that there exists a closed strict $\gamma$-cone
$A\subset (T^*I)\times \R$ such that
$\SSi(F) \subset T^*M \times \opb{\piw}(A)$.
Then, for any $s_1<s_2 \in I$, $\roim{q}(F\tenso \cor_{M\times [s_1,s_2[\times \R})$
and $\roim{q}(F\tenso \cor_{M\times ]s_1,s_2]\times \R})$ are torsion objects
of $\Derb_Z(\cor_{M\times\R})$.
\end{proposition}
\begin{proof}
(i) We only consider $G\eqdot \roim{q}(F\tenso \cor_{M\times [s_1,s_2[\times \R})$,
the other case being similar.
We may restrict ourselves to a relatively compact subinterval of $I$ containing
$s_1$ and $s_2$. Hence we may assume that $\SSi(F)$ is contained in
$T^*M \times \{ \tau\geq a |\sigma| \}$ for some $a>0$.  Then, applying
Lemma~\ref{lem:cone-chgt-coord} and changing $a$ if necessary, we may assume
that $I=\R$.

\medskip\noindent
(ii) We set $\alpha = a^{-1}$ so that
$\gamma_\alpha^\circ=\{ \tau\geq a |\sigma| \}$ and 
$\SSi(F)\subset T^*M \times \ol U_\alpha$.
Since $\SSi(\cor_{M\times [s_1,s_2[\times \R})
\subset T^*_MM\times T^*\R \times T^*_\R\R$, Corollary~\ref{cor:opboim}
gives $F\tenso \cor_{M\times [s_1,s_2[ \times \R}
\simeq \rsect_{M\times ]s_1,s_2] \times \R}(F)$ and
the formula~\eqref{eq:convgamma-id-ter} gives
$$
G \simeq \roim{q}\roim{s_E} \rsect_{M\times D} (\opb{q_1}F),
$$
where $D = (E\times \Int \gamma_\alpha) \cap \{(s,t,s',t'); s_1<s+s'\leq s_2 \}$.
We consider the commutative diagram
$$
\xymatrix@C=1.5cm{
&M\times E \times E \ar[r]^{s_E}\ar[d]^{\id_M\times \tilde q} \ar[dl]_{q_1}
&  M \times E \ar[d]^q  \\
M\times E  & M\times E \times \R \ar[r]^{\tilde s}\ar[l]_{\tilde q_1}
& M\times \R,  }
$$
where $\tilde q(s,t,s',t') = (s,t,t')$, $\tilde q_1(x,s,t,t') = (x,s,t)$
and $\tilde s(x,s,t,t') = (x,t+t')$.
The adjunction between $\reim{(\id_M\times \tilde q)}$ and
$\epb{(\id_M\times \tilde q)}$ gives
\eq
\notag
G & \simeq & \roim{\tilde s} \roim{(\id_M\times \tilde q)}
\rhom(\cor_{M\times D}, \epb{(\id_M\times \tilde q)} \opb{\tilde q_1}F)[-1] \\
\label{eq:proj-torsion}
& \simeq & \roim{\tilde s} 
\rhom(\cor_M\etens\reim{\tilde q} \cor_D, \opb{\tilde q_1}F)[-1] .
\eneq

\medskip\noindent
(iii) Through the isomorphism~\eqref{eq:convgamma-id-ter} the morphism
$\tau_c(F)$ is induced by the morphism
$\cor_{T_c(E\times \Int \gamma_\alpha)} \to \cor_{E\times \Int \gamma_\alpha}$, where
$T_c(s,t,s',t') = (s,t,s',t'+c)$.
Using~\eqref{eq:proj-torsion} it follows that $\tau_c(G)$ is induced 
by the morphism $u_c\cl \cor_{T_c(D)} \to \cor_D$. Hence it is enough to see
that the image of $u_c$ by $\reim{\tilde q}$ is the zero morphism.
In the remainder of the proof we show that $\reim{\tilde q} \cor_D$ and
$\reim{\tilde q} \cor_{T_c(D)}$ have disjoint supports for $c$ big enough.

\medskip\noindent
(iv) For a given point $(s,t,t') \in E\times\R$ we have
$\opb{\tilde q}(s,t,t') \cap D = \emptyset$ if $t'<0$ and otherwise
\eqn
\opb{\tilde q}(s,t,t') \cap D
& =& \{s';\; s_1-s <s'\leq s_2-s, \; t'\geq \alpha |s'|\}  \\
&= & ]s_1-s ,s_2-s] \cap [-\alpha^{-1} t', \alpha^{-1} t']  .
\eneqn
This is $\emptyset$ or a half closed interval 
when $t'$ is not in $I_s \eqdot [-\alpha(s_2-s), -\alpha(s_1-s)[$.
It follows that $\supp(\reim{\tilde q} \cor_D)$ is contained in
$D'\eqdot \{(s,t,t');\; t'\in \ol{I_s} \}$.
The support of $\reim{\tilde q} \cor_{T_c(D)}$ is contained in
$T'_c(D')$, with $T'_c(s,t,t') = (s,t,c+t')$.
Since $I_s$ is of length $\alpha(s_2-s_1)$ (independent of $s$) we obtain
$D'\cap T'_c(D') =\emptyset$ for $c> \alpha(s_2-s_1)$.
\end{proof}

From now on, we consider a connected manifold $M$ and
$F\in \Derb(\cor_{M\times\R})$.  We set $\gamma = M\times [0,+\infty[$ and we
make the hypothesis
\eq\label{eq:hyp-SSF-strict}
&&\mbox{$\SSi(F) \subset \opb{\piw}(A)$ for some closed $\gamma$-strict
cone $A\subset (T^*M)\times\R$.}
\eneq
In particular $F\in \Derb_{\{\tau\geq 0\}}(\cor_{M\times\R})$.
\begin{lemma}\label{lem:propag-of-torsion-local}
Let $F\in \Derb(\cor_{M\times\R})$ satisfying~\eqref{eq:hyp-SSF-strict}.
We assume that there exists $x \in M$ such that $F|_{\{x\}\times \R}$ is a
torsion object in $\Derb_{\{\tau\geq 0\}}(\cor_\R)$.  Then there exists a
neighborhood $U$ of $x$ such that $F|_{U\times \R}$ is a torsion object in
$\Derb_{\{\tau\geq 0\}}(\cor_{U\times\R})$.
\end{lemma}
\begin{proof}
(i) We may assume that $M$ is an open set in some vector space $V$ and $x=0$.
We take coordinates $(x,t;\xi,\tau)$ on $T^*(M\times \R)$. We may also assume
that $\SSi(F) \subset \{ \tau \geq a ||\xi|| \}$ for some $a>0$
and that $M$ contains the open ball of radius $1$, say $B$.
We set $I = ]-1,1[$ and take coordinates $(s;\sigma)$ on $T^*I$.
We define the  homotopy
$h\cl B\times I\times \R \to B\times \R$, $(x,s,t) \mapsto (sx,t)$.
For $s_0\in I$ we set $h_{s_0} = h(\cdot,s_0,\cdot)$.

\medskip\noindent (ii)
We check that $\opb{h}(F|_{B\times \R})$ satisfies the hypothesis of 
Proposition~\ref{prop:proj-torsion}.
We have $h_\pi(x,s,t;\xi,\tau) = (sx;t\xi,\tau)$ and
$h_d(x,s,t;\xi,\tau) = (x,s,t;s\xi, \langle x,\xi \rangle,\tau)$.
Hence $\ker h_d$ is contained in $\{\tau =0\}$.
Since $\SSi(F) \cap \{\tau =0\}$ is contained in the zero-section, $F$ is
non-characteristic for $h$ and we find
$$
\SSi(\opb{h}(F)) \subset  \{ (x',s',t';\xi',\sigma',\tau');\;
\sigma'= \langle x',\xi' \rangle, \; \tau' \geq a ||\xi'||/|s'| \}.
$$
On $B\times I$ we have $|s'|\leq 1$ and
$|\langle x',\xi' \rangle| \leq ||\xi'||$.  We deduce
$\SSi(\opb{h}(F)) \subset \{ \tau' \geq a |\sigma'| \}$ on
$B\times I\times \R$, as required.

\medskip\noindent (iii)
We apply Proposition~\ref{prop:proj-torsion} to $\opb{h}(F)$ on
$B\times I \times \R$ with $s_1 =0$, $s_2=1/2$.
For $J\subset I$ we set
$G_J=\roim{q}(\opb{h}(F|_{B\times \R})\tenso \cor_{M\times J\times \R})$.
We note that $G_{\{s\}} \simeq \opb{h_s}(F|_{B\times \R})$ for any $s\in I$.
We have the distinguished triangles on $B\times \R$
\eqn
G_{]0,1/2]} \to  G_{[0,1/2]} \to  G_{\{0\}}  \to[+1] , \qquad
 G_{[0,1/2[} \to  G_{[0,1/2]} \to  G_{\{1/2\}}  \to[+1] ,
\eneqn
where $G_{]0,1/2]}$ and $G_{[0,1/2[}$ are torsion by
Proposition~\ref{prop:proj-torsion}.
Since $h_0$ is the contraction $B\times \R \to \{0\}\times \R$ the hypothesis
implies that $G_{\{0\}}$ is torsion. Hence $G_{[0,1/2]}$ is torsion by the
first distinguished triangle and then $G_{\{1/2\}}$ also is torsion by the
second one. Since $h_{1/2}$ is a diffeomorphism from $B\times \R$ to 
$U\times \R$, where $U$ is the ball of radius $1/2$ we deduce that
$F|_{U\times\R}$ is torsion.
\end{proof}

\begin{lemma}\label{lem:propag-of-torsion-punctual}
Let $F\in \Derb(\cor_{M\times\R})$ satisfying~\eqref{eq:hyp-SSF-strict}.
We assume that there exists $x_0\in M$ such that
$F|_{\{x_0\}\times \R}$ is a torsion object in $\Derb_{\{\tau\geq 0\}}(\cor_\R)$.
Then $F|_{\{x\}\times \R}$ also is a torsion object in
$\Derb_{\{\tau\geq 0\}}(\cor_\R)$ for all $x\in M$.
\end{lemma}
\begin{proof}
We set $I = ]-1,1[$ and we choose an immersion $i\cl I \to M$ such that
$i(0) = x_0$ and $i(1/2) = x$. Then $\opb{i}F$ satisfies the hypothesis
of Proposition~\ref{prop:proj-torsion} on $I\times \R$.
We let $q\cl I\times \R \to \R$ be the projection.
Then $F|_{\{i(s)\}\times \R} \simeq \roim{q}(\opb{i}F\tenso \cor_{\{s\}\times \R})$
for any $s\in I$.
Now we have the distinguished triangles
\eqn
\roim{q}(\opb{i}F\tenso \cor_{]0,1/2] \times \R})  \to
\roim{q}(\opb{i}F\tenso \cor_{[0,1/2] \times \R})  \to
\opb{i}F|_{\{x_0\}\times \R}   \to[+1] , \\
\roim{q}(\opb{i}F\tenso \cor_{[0,1/2[ \times \R})  \to
\roim{q}(\opb{i}F\tenso \cor_{[0,1/2] \times \R})  \to
\opb{i}F|_{\{x\}\times \R}  \to[+1] 
\eneqn
and we conclude as in part (iii) of the proof of
Lemma~\ref{lem:propag-of-torsion-local}.
\end{proof}

\begin{theorem}\label{th:propag-of-torsion}
Let $M$ be a connected manifold and 
let $F\in \Derb(\cor_{M\times\R})$ satisfying~\eqref{eq:hyp-SSF-strict}.
Then the following assertions are equivalent:
\bnum
\item there exists $x_0\in M$ such that
$F|_{\{x_0\}\times \R}$ is a torsion object in $\Derb_{\{\tau\geq 0\}}(\cor_\R)$,
\item for any relatively compact open subset $U\subset M$ the restriction
$F|_{U\times \R}$ is a torsion object in $\Derb_{\{\tau\geq 0\}}(\cor_{U\times\R})$.
\enum
\end{theorem}
\begin{proof}
We only need to prove that (i) implies (ii).
By Lemmas~\ref{lem:propag-of-torsion-local}
and~\ref{lem:propag-of-torsion-punctual} we can find a finite cover of $\ol U$,
say $\{U_i\}$, $i=1,\ldots,n$, such that $F|_{U_i\times \R}$ is torsion.
We conclude with the remark that, for any two open subsets $V,W \subset M$,
if $F|_{V\times \R}$ and $F|_{W\times \R}$  are torsion,
then so is $F|_{(V\cup W)\times \R}$.
Indeed we apply Lemma~\ref{lem:critere-torsion} to the triangle
$F_{(V\cap W)\times \R} \to
F_{V\times \R} \oplus F_{W\times \R} \to 
F_{(V\cup W)\times \R} \to[+1]$ and the commutative square
$$
\xymatrix@C=2cm{
F_{(V\cap W)\times \R} \ar[r] \ar[d]_{\tau_{0,c} = 0}
& F_{V\times \R} \oplus F_{W\times \R}  \ar[d]^{\tau_{0,c} = 0} \ar[dl]_0  \\
\oim{T_c}(F_{(V\cap W)\times \R}) \ar[r] 
& \oim{T_c}(F_{V\times \R} \oplus F_{W\times \R}) .
}
$$
\end{proof}

\section{Tamarkin's non displaceability theorem}\label{section:tam3}
We will explain here Tamarkin's non displaceability theorem which gives a
criterion in order that two compact subsets of $T^*M$ are non displaceable.

In this section we consider a Hamiltonian isotopy
$\Phi\cl T^*M\times I\to T^*M$ satisfying~\eqref{eq:hyp-support-isot}, that
is, there exists a compact set $C\subset T^*M$ such that
$\phi_s|_{T^*M \setminus C}$ is the identity for all $s\in I$.

Let $\tw \Phi\cl \dT^*(M\times \R)\times I\to \dT^*(M\times\R)$ be the
homogeneous Hamiltonian isotopy given by Proposition~\ref{pro:homnonhomHIso}
and $\tw \Lambda \subset T^*(M\times \R \times M\times \R \times I)$ the conic
Lagrangian submanifold associated to $\tw \Phi$ in~\eqref{eq:def-lambda}.
Let $\tw K \in \Derlb(\cor_{M\times \R \times M\times \R \times I})$
be the quantization of $\tw \Phi$ given in Theorem~\ref{th:3}.

\subsubsection*{Invariance by Hamiltonian isotopy}
For $J\subset I$ a relatively compact subinterval of $I$, 
we introduce the kernel
\eqn
K^J = \reim{q_{1234}}(\tw K\tenso\cor_{M \times \R \times M\times \R \times J})
& \in  & \Derb(\cor_{M \times\R\times M\times \R}),
\eneqn
where $q_{1234}$ is the projection on the first four factors.
We remark that $\tw K$ and $K^J$ satisfy the hypothesis~\eqref{eq:hyp-kernel}.
Hence, by Proposition~\ref{prop:action-kernel},
composition with $K^J$ defines a functor
\eq\label{eq:def-Psi}
\Psi_J\cl \Derb(\cort_M) \to \Derb(\cort_M),
\qquad
F \mapsto K^J \conv F.
\eneq
We note that $K^{\{s\}} \simeq \tw K|_{M \times \R\times M\times \R \times \{s\}}$.
We set for short $\Psi_s = \Psi_{\{s\}}$. We have $\Psi_0 \simeq \id$.

\begin{theorem}\label{thm:invariance}
Let $\Phi\cl T^*M\times I\to T^*M$ be a Hamiltonian isotopy 
satisfying~\eqref{eq:hyp-support-isot}.
For $s\in I$ and $J\subset I$ a relatively compact subinterval
let $\Psi_J,\Psi_s\cl \Derb(\cort_M) \to \Derb(\cort_M)$ be the functors
defined in~\eqref{eq:def-Psi}.
Then for $A$ a closed subset of $T^*M$ and $F\in\Derb_{A}(\cort_M)$ we have
\bnum
\item 
$\Psi_s(F) \in \Derb_{\phi_s(A)}(\cort_M)$ for any $s\in I$,
\item 
$\Psi_{[a,b[}(F)$ and $\Psi_{]a,b]}(F)$ are torsion objects
for any $a<b \in I$,
\item 
for $s\in I$, $s\geq 0$, there exist distinguished triangles
$$
\Psi_{]0,s]}(F) \to  \Psi_{[0,s]}(F) \to F \to[+1],
\quad
\Psi_{[0,s[}(F) \to  \Psi_{[0,s]}(F) \to \Psi_s(F) \to[+1]
$$
and similar ones for $s\leq 0$.
In particular we have a natural isomorphism $F\simeq \Psi_s(F)$ 
in $\sht(\cor_M)$ for any $s\in I$.
\enum
\end{theorem}
\begin{proof}
(i) We set $\tw \Lambda_s = \tw \Lambda \circ T^*_sI$. This is the graph
of $\tw\phi_s$. Hence
$$
\SSi(\Psi_s(F))\cap \{\tau>0\} \subset \tw \Lambda_s \circ \opb{\rho}(A)
= \tw\phi_s(\opb{\rho}(A)) = \opb{\rho} (\phi_s(A)),
$$
which proves the first statement.

\medskip\noindent
(ii)-(iii)
(a) We set $\tw F = \tw K \circ F$ which belongs to $\Derlb(\cort_{M\times I})$
by Proposition~\ref{prop:action-kernel}.
We have $\SSi(\tw F) \cap \{\tau>0\} \subset \tw \Lambda \circ \opb{\rho}(A)$.
As in Lemma~\ref{lem:Acompact}
we define $A' \subset T^*M \times \R$ by
$A' = \{(x;\xi,\tau); \tau>0, \; (x;\xi/\tau) \in A\}$.
Then $\ol{A'}$ is a strict $\gamma$-cone.
It follows that there exists a closed strict  $\gamma$-cone
$B\subset T^*(M\times I) \times \R$  such that
$\tw \Lambda \circ \opb{\rho}(A) \subset \opb{\piw}(B) \cap \{\tau>0\}$.
Then Lemma~\ref{lem:Acompact} gives
$\SSi(\tw F) \subset \opb{\piw}(B)
\cup T^*_{M\times I\times \R}(M\times I\times \R)$.
In particular $\tw F|_{M\times J\times \R}$ satisfies the hypothesis of
Proposition~\ref{prop:proj-torsion} for any relatively compact subinterval
$J\subset I$.

\medskip\noindent
(b) We let $q\cl M\times I \times \R \to M\times \R$ be the projection.
For a relatively compact subinterval $J\subset I$ we have
$\Psi_J(F) \simeq \roim{q}(\tw F\tenso \cor_{M\times J\times \R})$.
Then (ii) follows from Proposition~\ref{prop:proj-torsion}.
The triangles in (iii) are induced by the excision triangles associated with
the inclusions $\{0\} \subset [0,s]$ and $\{s\} \subset [0,s]$.
Then (ii) gives $F\isofrom \Psi_{[0,s]}(F) \isoto \Psi_s(F)$ in $\sht(\cor_M)$.
\end{proof}

\subsubsection*{Application to non displaceability}
Recall that two compact subsets $A$ and $B$ of $T^*M$ are called mutually non
displaceable if, for any Hamiltonian isotopy
$\Phi\cl T^*M\times I\to T^*M$ satisfying~\eqref{eq:hyp-support-isot}
and any $s\in I$,
$A\cap\phi_s(B)\not=\emptyset$. A compact subset $A$ is called non
displaceable if $A$ and $A$ are mutually non displaceable.
Let $A$ and $B$ be two compact subsets of $T^*M$, let 
$F\in\Derb_A(\corgl_M)$ and $G\in\Derb_B(\corgl_M)$.
Let $q_2\cl M\times\R \to \R$ be the projection.  Recall that
$\rhomc(F,G) \in\Derb(\corgr_M)$ by~\eqref{eq:rhomcperpr}. We deduce by
adjunction that $\roim{q_2}\rhomc(F,G) \in\Derb(\corgr)$.
We shall consider the following hypothesis:
\eq\label{hyp:nondisp1}
&& \mbox{$\roim{q_2}\rhomc(F,G)$ is not torsion.}
\eneq

\begin{theorem}\label{th:nondis1}
{\rm (The non displaceability Theorem of~\cite[Th.~3.1]{Ta}.)}
Let $A$ and $B$ be two compact subsets of $T^*M$. Assume that 
there exist $F\in\Derb_A(\corgl_M)$ and $G\in\Derb_B(\corgl_M)$
satisfying the hypothesis~\eqref{hyp:nondisp1}.
Then $A$ and $B$ are mutually non displaceable in $T^*M$.
\end{theorem}
\begin{proof}
Assume $\Phi$ is a Hamiltonian isotopy such that
$\phi_{s_0}(B)\cap A=\emptyset$. 
We consider $\tw \Phi\cl \dT^*(M\times \R)\times I\to \dT^*(M\times\R)$ 
and $\tw K \in \Derlb(\cor_{M\times \R \times M\times \R \times I})$ as in the
introduction of this section.

We define $F',G'\in \Derb(\corgl_{M\times I})$ by
$F' = F\etens\cor_I$ and $G' = \tw K \circ G$. We let
$q_{23}\cl M\times \R \times I \to \R \times I$ be the projection.
We have $F \simeq F'|_{M \times \R \times \{s\}}$ and we set 
$G_s = G'|_{M \times \R \times \{s\}}$. By Lemma~\ref{lem:Acompact} and
Corollary~\ref{cor:restr-rhomc}, we have
$\rhomc(F',G')|_{M \times \R \times \{s\}} \simeq \rhomc(F,G_s)$.
By Lemma~\ref{lem:Acompact} 
 $q_{23}$ is proper on the support of $\rhomc(F',G')$ and we get
$$
(\roim{q_{23}}\rhomc(F',G'))|_{M \times \R \times \{s\}}
\simeq \roim{q_2}\rhomc(F,G_s).
$$
Since $\SSi(G_s) \subset \opb{\rho}(\phi_s(B))$, 
Theorem~\ref{th:Tamvan} implies $\roim{q_2}\rhomc(F,G_{s_0}) = 0$.

By Proposition~\ref{pro:SSconv} and Lemma~\ref{lem:star-gamma-cones},
the microsupport of $\rhomc(F',G')$ is contained in $\opb{\piw}(C)$
for some strict $\gamma$-cone $C$. Hence a similar inclusion holds for the
microsupport of $\roim{q_{23}}\rhomc(F',G')$.
Then Theorem~\ref{th:propag-of-torsion} implies that $\rhomc(F,G_s)$ is
torsion for all $s\in I$. In particular $\rhomc(F,G)$ is torsion, which
contradicts the hypothesis~\eqref{hyp:nondisp1}.
\end{proof}

\begin{corollary}\label{cor:sht-nondisp1}
Let $A$ and $B$ be two compact subsets of $T^*M$. Assume that 
there exist $F\in\Derb_A(\cort_M)$ and $G\in\Derb_B(\cort_M)$
such that $\Hom[\sht(\cor_M)](F,G) \not= 0$.
Then $A$ and $B$ are mutually non displaceable in $T^*M$.
\end{corollary}
\begin{proof}
By Proposition~\ref{pro:locandlim}, there exists $c\in \R$ such that
the morphism induced by $\tau_{c,d}(G)$,
$\Hom[{\Derb(\cort_M)}](F,\oim{T_c} G)
\to \Hom[{\Derb(\cort_M)}](F,\oim{T_d} G)$
is non zero  for all $d\geq c$.
But Lemma~\ref{le:secthomc} gives
$$
\Hom[{\Derb(\cort_M)}](F,\oim{T_c} G) \simeq
H^0_{[0,+\infty[}(\R;\roim{q_2}\rhomc(F, \oim{T_c} G)).
$$
On the other hand we see that
$\roim{q_2}\rhomc(F, \oim{T_c} G) \simeq \oim{T_c}\roim{q_2}\rhomc(F, G)$
and that $\tau_{c,d}(G)$ induces $\tau_{c,d}(\roim{q_2}\rhomc(F, G))$
through this isomorphism.
Hence $\roim{q_2}\rhomc(F, G)$ is non torsion and we can apply
Theorem~\ref{th:nondis1}.
\end{proof}
Let $A$ be a closed conic subset of $T^*M$. We know by
Corollary~\ref{cor:dbMtodbE} that the functor
\eq\label{eq:fctff}
&&j_M\cl \Derb_A(\cor_M)\to\sht(\cor_M),
\quad F\mapsto F\etens\cor_{[0,+\infty[}
\eneq
is fully faithful.
Applying Corollary~\ref{cor:sht-nondisp1} with
$F=G=j_M(\cor_M)\in\sht(\cor_M)$ and $A=B = T^*_MM$, we get
\begin{corollary}\label{cor:nondis2}
Assume $M$ is compact. Then  $M$ is non displaceable in $T^*M$.
\end{corollary}

In~\cite{Ta}, Tamarkin applies the non displaceability Theorem~\ref{th:nondis1} to
prove that the following sets are non displaceable.

Set $X=\BBP(\C)^n$ endowed with his standard real symplectic
structure. Consider the sets $A\eqdot \BBP(\R)^n$ and 
$B\eqdot \BBT=\{z=(z_0,\dots,z_n);\vert{z_0}\vert=\dots\vert{z_n\vert}\}$.
Then $A$ and $B$ are non displaceable and $A$ and $B$ are mutually non
displaceable.

\providecommand{\bysame}{\leavevmode\hbox to3em{\hrulefill}\thinspace}

\vspace*{1cm}
\noindent
\parbox[t]{16em}
{\scriptsize{
\noindent
St{\'e}phane Guillermou\\
Institut Fourier\\
Universit{\'e} de Grenoble I\\
BP 74, 38402
Saint-Martin d'H{\`e}res, France\\
email: Stephane.Guillermou@ujf-grenoble.fr\\
http://www-fourier.ujf-grenoble.fr/\textasciitilde guillerm/}
}
\hspace{0.1cm}
\parbox[t]{16em}
{\scriptsize{
Pierre Schapira\\
Institut de Math{\'e}matiques\\
Universit{\'e} Pierre et Marie Curie\\
4, place Jussieu, case 247, 75252 Paris cedex 5, France \\
e-mail: schapira@math.jussieu.fr\\
http://www.math.jussieu.fr/\textasciitilde schapira/}}

\end{document}